\documentclass[11pt,a4paper]{article}
\usepackage{latexsym}
\usepackage{amsmath}
\usepackage{amssymb}
\usepackage{amscd}
\usepackage{amsthm}
\usepackage[all]{xy}

\newtheorem{thm}{Theorem}[section]
\newtheorem{cor}[thm]{Corollary}
\newtheorem{lem}[thm]{Lemma}
\newtheorem{prop}[thm]{Proposition}

\theoremstyle{definition}
\newtheorem{defn}[thm]{Definition}
\newtheorem{rem}[thm]{Remark}
\newtheorem{exa}[thm]{Example}

\newcommand{\Q}{\mathbb Q}
\newcommand{\R}{\mathbb R}
\newcommand{\Z}{\mathbb Z}

\newcommand{\C}{\mathbb C}

\newif\ifpdf \pdftrue
\ifx\pdfoutput\undefined \pdffalse \fi \ifx\pdfoutput\relax
\pdffalse \fi

\usepackage{makeidx}
\makeindex
\begin{document}

\title{Localization theorem for higher arithmetic K-theory}

\author{Shun Tang}

\date{}

\maketitle

\vspace{-10mm}

\hspace{5cm}\hrulefill\hspace{5.5cm} \vspace{5mm}

\textbf{Abstract.} Quillen's localization theorem is well known as a fundamental theorem in the study of algebraic K-theory. In this paper, we present its arithmetic analogue for the equivariant K-theory of arithmetic schemes, which are endowed with an action of certain diagonalisable group scheme. This equivariant arithmetic K-theory is defined by means of a natural extension of Burgos-Wang's simplicial description of Beilinson's regulator map to the equivariant case. As a byproduct of this work, we give an analytic refinement of the Riemann-Roch theorem for higher equivariant algebraic K-theory. And as an application, we prove a higher arithmetic concentration theorem which generalizes Thomason's corresponding result in purely algebraic case to the context of Arakelov geometry.


%\newpage
%\vfill
\tableofcontents
%\vfill
%\null
%\clearpage

\section{Introduction}
\label{intro}
Higher algebraic K-theory has been constructed for any exact category $\mathcal{E}$. By means of Quillen's $Q$-construction, the $m$-th algebraic K-group of $\mathcal{E}$ is defined as the $(m+1)$-th homotopy group of the classifying space of $Q\mathcal{E}$, i.e. $$K_m(\mathcal{E}):=\pi_{m+1}(BQ\mathcal{E}).$$

Specifying the case of algebraic K-theory of schemes, let $X$ be a scheme (which means a noetherian, separated scheme of finite type over an affine noetherian scheme), we denote by $\mathcal{M}(X)$ (resp. $\mathcal{P}(X)$) the exact category of coherent sheaves (resp. locally free sheaves of finite rank) on $X$. Then the algebraic $G$-groups (resp. $K$-groups) of $X$ are defined to be $$G_m(X):=K_m(\mathcal{M}(X))\quad\big(\mathrm{resp. }K_m(X):=K_m(\mathcal{P}(X))\big).$$

Such groups $G_m(X)$ and $K_m(X)$ contain abundant information about $X$. For instance, the motivic cohomology deduced from the Adams operations on $K_m(X)$ can be viewed as some universal cohomology theory and it provides a method to define the intersection product of algebraic cycles when $X$ is regular.

To study the properties of algebraic K-theory of schemes, Quillen's localization theorem is useful and elementary. Combined with the theorem \textit{D\'{e}vissage} (also due to Quillen), it usually appears as the following form:

\begin{thm}\label{tha}
Let $X$ be a scheme, $Y\subset X$ a closed subscheme and $U=X\setminus Y$ its complement. Then there is a long exact sequence of $G$-groups
$$\cdots\to G_m(Y)\to G_m(X)\to G_m(U)\to G_{m-1}(Y)\to\cdots.$$
\end{thm}

If $X$ and $Y$ are both regular, the algebraic $G$-groups and $K$-groups are isomorphic, so we have a long exact sequence
$$\cdots\to K_m(Y)\to K_m(X)\to K_m(U)\to K_{m-1}(Y)\to\cdots$$
which has more uses.

Now, denote by $\mu_n$ the diagonalisable group scheme associated to an acyclic group of order $n$, and let $X$ be a scheme endowed with an action of $\mu_n$, we may consider the exact category of all coherent or all locally free coherent sheaves on $X$ which admit compatible $\mu_n$-structures. Then Quillen's $Q$-construction leads to a definition of equivariant algebraic K-theory $G_*(X,\mu_n)$ and $K_*(X,\mu_n)$. This equivariant K-theory has been systematically studied by R. Thomason in \cite{Th1}, it has most of the same properties with the algebraic K-theory of schemes in the non-equivariant case. In particular, we have the localization theorem:

\begin{thm}\label{thb}
Let $X$ be a regular scheme endowed with a $\mu_n$-action, $Y\subset X$ a regular equivariant closed subscheme and $U=X\setminus Y$ its complement. Then there is a long exact sequence of equivariant $K$-groups
$$\cdots\to K_m(Y,\mu_n)\to K_m(X,\mu_n)\to K_m(U,\mu_n)\to K_{m-1}(Y,\mu_n)\to\cdots.$$
\end{thm}

In this paper, we focus on the same type theorem in the context of Arakelov geometry, we shall prove an arithmetic analogue of Theorem~\ref{thb} and hence of Theorem~\ref{tha} for the equivariant arithmetic K-theory.

Let us explain this more precisely. Suppose that $(D, \Sigma, F_\infty)$ is a regular arithmetic ring in the sense of Gillet-Soul\'{e} (cf. \cite{GS1}) and denote $\mu_n:={\rm Spec}(D[\Z/{n\Z}])$ the diagonalisable group scheme over $D$ associated to an acyclic group $\Z/{n\Z}$. By an equivariant arithmetic scheme $X$ over $D$, we understand a $\mu_n$-quasi-projective scheme over $D$ with smooth generic fibre. It is equivalent to say that $X$ is an arithmetic scheme over $D$ and there exists a very ample invertible $\mu_n$-sheaf on $X$. When we say $X$ is $\mu_n$-projective, that means $X$ is a projective arithmetic scheme equipped with a very ample invertible $\mu_n$-sheaf. The arithmetic $\mathrm{K}_0$-theory of the equivariant and the non-equivariant arithmetic schemes has been studied in several literatures (cf. \cite{GS1}, \cite{GS2}, \cite{KR} and \cite{T1}). In these literatures one common thing was defining the arithmetic $\mathrm{K}_0$-group $\widehat{K}_0(X,\mu_n)$ as a modified Grothendieck group of equivariant hermitian vector bundles on $X$. This group $\widehat{K}_0(X,\mu_n)$ and its character maps to the arithmetic Chow ring evolved many interesting structure statements in Arakelov geometry, such as the intersection theory, the Riemann-Roch theorem and the Lefschetz fixed point formula.

So far, the higher arithmetic K-theory is rarely touched. However, Soul\'{e} (cf. \cite{So}), and also Deligne (cf. \cite{De}), suggested that the higher arithmetic K-groups of $X$ can be defined as the homotopy groups of the homotopy fibre of Beilinson's regulator map. That means we shall obtain a long exact sequence
$$\xymatrix{\cdots \ar[r] & \widehat{K}_m(X) \ar[r] & K_m(X) \ar[r]^-{{\rm ch}} & \bigoplus_{p\geq 0}H_{\mathcal{D}}^{2p-m}(X,\R(p)) \ar[r] & \widehat{K}_{m-1}(X) \ar[r] & \cdots}$$
where $H_{\mathcal{D}}^{*}(X,\R(p))$ is the real Deligne-Beilinson cohomology and ${\rm ch}$ is the Beilinson's regulator map. We caution the reader that for non-proper $X$, ${\rm ch}$ is actually constructed using Beilinson's real absolute Hodge cohomology which is a refinement of Deligne cohomology. In this paper, for the sake of clear expression, we shall use the weak form but all arguments we present are valid for Beilinson's real absolute Hodge cohomology.

To give a more concrete definition of $\widehat{K}_m(X)$ i.e. to get the homotopy fibre, a simplicial description of Beilinson's regulator map is necessary. In \cite{BW}, Burgos and Wang used the theory of higher Bott-Chern forms to carry out this work. They provided a functorial simplicial set $\widehat{S}(X)$, a functorial complex of vector spaces $\bigoplus_{p\geq0}\widetilde{D}^{2p-*}(X,p)[-1]$ and a canonical simplicial map $\widetilde{\rm ch}: \widehat{S}(X)\to \mathcal{K}(\bigoplus_{p\geq0}\widetilde{D}^{2p-*}(X,p)[-1])$ such that $\pi_{m+1}(\mid \widehat{S}(X)\mid)=K_m(X)$, $\pi_{m+1}(\mid\mathcal{K}(\bigoplus_{p\geq0}\widetilde{D}^{2p-*}(X,p)[-1])\mid)=\bigoplus_{p\geq0}H_{\mathcal{D}}^{2p-m}(X, \R(p))$ and $\pi_{m+1}(\mid\widetilde{\rm ch}\mid)={\rm ch}$. So $\widehat{K}_m(X)$ can be defined as $\pi_{m+1}(\text{homotopy fibre of}\mid\widetilde{\rm ch}\mid)$. Here the symbol $\mid\cdot\mid$ stands for the geometric realization of a simplicial set, $\mathcal{K}(\cdot)$ is the Dold-Puppe functor associating a simplicial abelian group to a homological complex of abelian groups. 

To $\mu_n$-equivariant arithmetic K-theory, we may follow a similar approach to get a feasible definition. The kernel is that a $\mu_n$-equivariant vector bundle on the fixed point subscheme $X_{\mu_n}$ with $\mu_n$-invariant hermitian metric orthogonally splits into a direct sum of its eigenbundles, and the definition of higher Bott-Chern form naturally extends to this case (see Definition~\ref{215} below). So we start with a simplicial set $\widehat{S}(X)$ which is associated to the category of $\mu_n$-equivariant vector bundles on $X$ with smooth at infinity metrics (not necessarily $\mu_n$-invariant), and pass to the fixed point subscheme by pull-back map $\varphi: \widehat{S}(X)\to \widehat{S}(X_{\mu_n})$. Then we make the $\mu_n$-average of any metric so that it becomes $\mu_n$-invariant and the theory of equivariant higher Bott-Chern forms can be applied. These manipulations deduce two simplicial maps $\vartheta: \widehat{S}(X_{\mu_n})\to \widehat{S}(X_{\mu_n}, \mu_n)$ and $\phi: \widehat{S}(X_{\mu_n}, \mu_n)\to \mathcal{K}(\bigoplus_{p\geq0}\widetilde{D}^{2p-*}(X_{\mu_n},p)[-1])_{R_n}$, where $\widehat{S}(X_{\mu_n}, \mu_n)$ is the simplicial set associated to the category of $\mu_n$-equivariant vector bundles on $X_{\mu_n}$ with $\mu_n$-invariant and smooth at infinity metrics, $R_n$ is $\R$ or $\C$ which contains the eigenvalues of $\mu_n$-structures. Composing $\varphi$, $\vartheta$ and $\phi$, we get a simplicial map $\widetilde{{\rm ch}_g}: \widehat{S}(X)\to \mathcal{K}(\bigoplus_{p\geq0}\widetilde{D}^{2p-*}(X_{\mu_n},p)[-1])_{R_n}$ which describes the equivariant regulators (see Definition~\ref{221}, ~\ref{222} below). Therefore, we may define $\widehat{K}_{*}(X,\mu_n):=\pi_{*+1}(\text{homotopy fibre of}\mid\widetilde{{\rm ch}_g}\mid)$.
 
The main result of this paper is the following:

\begin{thm}\label{thc}
Let $X$ be a regular $\mu_n$-equivariant arithmetic scheme which is proper over $D$, and let $Y\subset X$ be a regular equivariant arithmetic closed subscheme with $U=X\setminus Y$ its complement. Then there is a long exact sequence of equivariant arithmetic $K$-groups
$$\cdots\to \widehat{K}_m(Y,\mu_n)\to \widehat{K}_m(X,\mu_n)\to \widehat{K}_m(U,\mu_n)\to \widehat{K}_{m-1}(Y,\mu_n)\to\cdots$$
ending with
$$\cdots\to \widehat{K}_1(Y,\mu_n)\to \widehat{K}_1(X,\mu_n)\to \widehat{K}_1(U,\mu_n).$$
\end{thm}

The difficulty to the formulation and the proof of Theorem~\ref{thc} is that there is no evident exact functor from $\mathcal{P}(Y,\mu_n)$ to $\mathcal{P}(X,\mu_n)$ which connects corresponding K-theory spaces. Our strategy is to construct a middle-bridge, the equivariant arithmetic K-theory with supports $\widehat{K}_{Y,m}(X,\mu_n)$, so that Theorem~\ref{thc} follows from a natural long exact sequence
\begin{align}\label{a1}
\cdots\to \widehat{K}_{Y,m}(X,\mu_n)\to \widehat{K}_m(X,\mu_n)\to \widehat{K}_m(U,\mu_n)\to \widehat{K}_{Y,m-1}(X,\mu_n)\to\cdots
\end{align}
together with an arithmetic purity theorem
$$\widehat{K}_m(Y,\mu_n)\cong \widehat{K}_{Y,m}(X,\mu_n).$$

As a byproduct of the arithmetic purity theorem, we shall obtain an equivariant generalization of Gillet's Riemann-Roch theorem for closed immersions. Moreover, combining with a naive imitation of Roessler's analytic proof of Gillet's Riemann-Roch theorem for compact fibrations (cf. \cite{Roe}), we get an analytic refinement of the Riemann-Roch theorem for higher equivariant K-theory in case of projective morphisms. Such an analytic refinement allows us to extend the Lefschetz trace formula to the higher algebraic K-theory. Another proof of this Lefschetz trace formula is given by K\"{o}ck in \cite{Ko}.

The motivation of our work in this paper is that we want to prove an arithmetic concentration theorem which generalizes Thomason's corresponding result (cf. \cite{Th2}) to the context of Arakelov geometry. This theorem states that, after a suitable localization, the equivariant arithmetic K-group  $\widehat{K}_m(X_{\mu_n},\mu_n)_\rho$ is isomorphic to $\widehat{K}_m(X,\mu_n)_\rho$. Such a statement is considered the foundation of a proof of the higher arithmetic Lefschetz-Riemann-Roch theorem, as the $0$-degree case settled by the author in \cite{T1}. Further details should be discussed in other papers.

Next, it is worth pointing out that a variant formalism of the non-equivariant case $(n=1)$ of Theorem~\ref{thc} has been implied in two papers of A. Holmstrom and J. Scholbach, which concern a development of the Arakelov motivic cohomology (cf. \cite{HS} and \cite{S1}). In these two papers, the six functors formalism in motivic stable homotopy theory was used so that one needs not to construct explicit homotopies between complexes representing the regulator maps whenever a geometric construction (e.g. the pushforward) is to be done. However, to the aims that to get an intrinsic definition of the equivariant regulator maps, to make a complete extension of Theorem~\ref{thb} in Arakelov geometry and to formulate an arithmetic Lefschetz-Riemann-Roch theorem for higher arithmetic K-groups, the Bisumt's analytic machinery that we use in the present paper seems necessary.

Finally, we remark that Takeda provided in \cite{Ta} another way to define the higher arithmetic K-theory for proper arithmetic schemes, which also admits a natural generalization to the equivariant case. Nevertheless, we don't deal with the possible localization theorem for Takeda's arithmetic K-groups in this paper.

The structure of this paper is as follows. In Section~\ref{sec:2}, we construct the equivariant higher arithmetic K-groups following Soul\'{e} and Deligne, as an opportunity, we recall Burgos-Wang's higher Bott-Chern forms and the simplicial description of the regulator maps. In Section~\ref{sec:3}, we introduce the arithmetic K-theory with supports, the Chern character maps with supports and deduce the long exact sequence~(\ref{a1}). In Section~\ref{sec:4}, we state and prove the arithmetic purity theorem, which fulfills the proof of Theorem~\ref{thc}. In the last section, Section~\ref{sec:5}, we give a proof of the arithmetic concentration theorem as an application.

\section{Equivariant higher arithmetic K-groups}
\label{sec:2}
\subsection{Deligne-Beilinson cohomology and Deligne homology}
\label{sec:2.1}
The real Deligne-Beilinson cohomology roughly measures how the natural real structure on the singular cohomology is behaved with respect to the Hodge filtration. It is the target receiving Beilinson's regulator map and plays a main role in Beilinson's formulation of a series of deep conjectures on the special values of $L$-functions of algebraic varieties. In this subsection, we recall the definition and some properties of the real Deligne-Beilinson cohomology for smooth, not necessarily proper algebraic varieties over $\C$. We also recall its homological counterpart, the Deligne homology, and corresponding Poincar\'{e} duality. The main references for this subsection are \cite{Bu1} and \cite{BKK}.

Let $X$ be a smooth algebraic variety over $\C$, in this subsection, we shall work with the analytic topology of $X$. Let $I$ be the category of all smooth compactifications of $X$. The objects in $I$ are pairs $(\widetilde{X}_\alpha,\imath_\alpha)$ where $\widetilde{X}_\alpha$ is a smooth proper variety with immersion $\imath_\alpha: X\hookrightarrow\widetilde{X}_\alpha$ and $D_\alpha:=
\widetilde{X}_\alpha\setminus\imath_\alpha(X)$ is a normal crossing divisor. The morphisms in $I$ are maps $f: \widetilde{X}_\alpha\to \widetilde{X}_\beta$ such that $f\circ\imath_\alpha=\imath_\beta$. It can be shown that the opposed category $I^\circ$ is directed.

\begin{defn}\label{201}
The complex of differential forms on $X$ with logarithmic singularities along infinity is $$E_{\log}^*(X):=\lim\limits_{\stackrel{\longrightarrow}{\alpha\in I^\circ}}E_{\widetilde{X}_\alpha}^*(\log D_\alpha),$$ where $E_{\widetilde{X}_\alpha}^*(\log D_\alpha)$ stands for the complex of differential forms on $\widetilde{X}_\alpha$ with logarithmic singularities along $D_\alpha$ (for details, see \cite{Bu1}).
\end{defn}

Let $E_X^*$ be the complex of smooth differential forms on $X$, then $E_{\log}^*(X)$ is a subcomplex of $E_X^*$. Moreover, $E_{\log}^*(X)$ admits a Hodge filtration $$F^pE_{\log}^n(X)=\bigoplus_{\stackrel{p'\geq p}{p'+q'=n}}E_{\log}^{p',q'}(X)$$ induced by the natural bigrading $E_{\log}^*(X)=\oplus E_{\log}^{p,q}(X)$, and all the morphisms $$(E_{\widetilde{X}_\alpha}^*(\log D_\alpha),F)\to(E_{\log}^*(X),F)$$ are filtered quasi-isomorphisms.

Let us write $$\mathfrak{D}^*(X,p):=s\big(u: E_{\log,\R}^*(X,p)\oplus F^pE_{\log}^*(X)\to E_{\log}^*(X)\big)$$
where $E_{\log,\R}^*(X,p):=(2\pi i)^pE_{\log,\R}^*(X)$ with $E_{\log,\R}^*(X)$ the subcomplex of $E_{\log}^*(X)$ consisting of real forms, $u(a,b)=b-a$, and $s(\cdot)$ is the simple complex associated to a morphism of complexes i.e. $s(\cdot)=\text{cone}(\cdot)[-1]$ which is the mapping cone shifted by $1$. The differential of the complex $\mathfrak{D}^*(X,p)$ will be denoted by $d_\mathfrak{D}$.

The real Deligne-Beilinson cohomology of $X$ is $$H_{\mathcal{D}}^n(X,\R(p))=H^n(\mathfrak{D}^*(X,p)).$$

There is a simpler complex that can be used to compute the real Deligne-Beilinson cohomology, this complex allows us to represent a class in the real Deligne-Beilinson cohomology by a single differential form. We summarize the main results we need as follows and refer to \cite[Section 2]{Bu1} for the general theory of Dolbeault complex and the associated Deligne-Beilinson complex.

Firstly, for an element $x$ in $E_{\log}^n(X)=\bigoplus_{\stackrel{p+q=n}{}}E_{\log}^{p,q}(X)$, we define
$$F^{k,k}x=\sum_{\stackrel{l\geq k,l'\geq k}{}}x^{l,l'}\quad\text{and}\quad F^{k}x=\sum_{\stackrel{l\geq k}{}}x^{l,l'}.$$

On the other hand, the isomorphism $\C=\R(p)\oplus\R(p-1)$ induces a decomposition $E_{\log}^*(X)=E_{\log,\R}^*(X,p)\oplus E_{\log,\R}^*(X,p-1)$, and the projection $\pi_p: E_{\log}^*(X)\to E_{\log,\R}^*(X,p)$ is given by $$\pi_p(x)=\frac{1}{2}\big(x+(-1)^p\overline{x}\big).$$ Then we set $\pi(x):=\pi_{p-1}(F^{n-p+1,n-p+1}x)$.

\begin{thm}\label{202}
Set $$\mathfrak{D}^n(E_{\log}(X),p)=\left\{
      \begin{array}{ll}
        E_{\log,\R}^{n-1}(X,p-1)\bigcap\bigoplus_{\stackrel{p'+q'=n-1}{p'<p,q'<p}}E_{\log}^{p',q'}(X), & n<2p; \\
        E_{\log,\R}^{n}(X,p)\bigcap\bigoplus_{\stackrel{p'+q'=n}{p'\geq p,q'\geq p}}E_{\log}^{p',q'}(X), & n\geq 2p,
      \end{array}
    \right.$$
with differential
$$d_\mathcal{D}x=\left\{
                   \begin{array}{ll}
                     -\pi(dx), & n<2p-1; \\
                     -2\partial\overline{\partial}x, & n=2p-1; \\
                     dx, & n>2p-1.
                   \end{array}
                 \right.$$
Then

(i). the complexes $\mathfrak{D}^*(X,p)$ and $\mathfrak{D}^*(E_{\log}(X),p)$ are homotopically equivalent. The homotopy equivalences $\psi: \mathfrak{D}^*(X,p)\to \mathfrak{D}^*(E_{\log}(X),p)$ and $\phi: \mathfrak{D}^*(E_{\log}(X),p)\to \mathfrak{D}^*(X,p)$ are given by
$$\psi(a,f,\omega)=\left\{
                     \begin{array}{ll}
                       \pi(\omega), & n\leq 2p-1; \\
                       F^{p,p}a+2\pi_p(\partial\omega^{p-1,n-p+1}), & n\geq 2p,
                     \end{array}
                   \right.$$
and
$$\phi(x)=\left\{
            \begin{array}{ll}
              (\partial x^{p-1,n-p}-\overline{\partial}x^{n-p,p-1},2\partial x^{p-1,n-p},x), & n\leq 2p-1; \\
              (x,x,0), & n\geq 2p.
            \end{array}
          \right.$$
Moreover $\psi\phi={\rm Id}$ and $\phi\psi-{\rm Id}=dh+hd$, where $h: \mathfrak{D}^n(X,p)\to \mathfrak{D}^{n-1}(X,p-1)$ is given by
$$h(a,f,\omega)=\left\{
                  \begin{array}{ll}
                    (\pi_p(\overline{F}^p\omega+\overline{F}^{n-p}\omega),-2F^p(\pi_{p-1}\omega),0), & n\leq 2p-1; \\
                    (2\pi_p(\overline{F}^{n-p}\omega),-F^{p,p}\omega-2F^{n-p}(\pi_{p-1}\omega),0), & n\geq 2p.
                  \end{array}
                \right.$$

(ii). The natural morphism $H^*(\mathfrak{D}^*(X,p))\to H^*(E_{\log,\R}^*(X,p))$ is induced by the morphism of complexes
$$r_p: \mathfrak{D}^*(E_{\log}(X),p)\to E_{\log,\R}^*(X,p)$$ given by
$$r_px=\left\{
         \begin{array}{ll}
           2\pi_p(F^pdx)=\partial x^{p-1,n-p}-\overline{\partial}x^{n-p,p-1}, & n\leq 2p-1; \\
           x, & n\geq 2p.
         \end{array}
       \right.$$
\end{thm}
\begin{proof}
This is \cite[Theorem 2.6]{Bu1}.
\end{proof}

\begin{cor}\label{203}
Let $X$ be a smooth algebraic variety over $\C$, then
$$H_{\mathcal{D}}^n(X,\R(p))=H^n(\mathfrak{D}^*(E_{\log}(X),p)).$$
We shall write $D^*(X,p):=\mathfrak{D}^*(E_{\log}(X),p)$ for short.
\end{cor}

\begin{rem}\label{204}
(i). The real Deligne-Beilinson cohomology of $X$ at degrees $2p$ and $2p-1$ are given by
$$H^{2p}(\mathfrak{D}^*(E_{\log}(X),p))=\{x\in E_{\log}^{p,p}(X)\cap E_{\log,\R}^{2p}(X,p)\mid dx=0\}/{{\rm Im}(\partial\overline{\partial})}$$
and
$$H^{2p-1}(\mathfrak{D}^*(E_{\log}(X),p))=\{x\in E_{\log}^{p-1,p-1}(X)\cap E_{\log,\R}^{2p-2}(X,p-1)\mid \partial\overline{\partial}x=0\}/{({\rm Im}\partial+{\rm Im}\overline{\partial})}.$$

(ii). Replace $E_{\log}^*(X)$ by $E_X^*$, one gets a complex $\mathfrak{D}^*(E_X,p)$ whose cohomology is called the analytic Deligne cohomology of $X$. When $X$ is proper, it is isomorphic to the real Deligne-Beilinson cohomology.

(iii). Let $x\in \mathfrak{D}^n(E_{\log}(X),p)$ and $y\in \mathfrak{D}^m(E_{\log}(X),q)$, we write $l=n+m$ and $r=p+q$. Then
$$x\bullet y=\left\{
               \begin{array}{ll}
                 (-1)^nr_p(x)\wedge y+x\wedge r_q(y), & n<2p, m<2q; \\
                 \pi(x\wedge y), & n<2p, m\geq 2q, l<2r; \\
                 F^{r,r}(r_p(x)\wedge y)+2\pi_r\partial((x\wedge y)^{r-1,l-r}), & n<2p, m\geq 2q, l\geq 2r; \\
                 x\wedge y, & n\geq 2p, m\geq 2q.
               \end{array}
             \right.$$
induces an associative and commutative product in the real Deligne-Beilinson cohomology which coincides with the product defined by Beilinson. Moreover, if $x\in \mathfrak{D}^{2p}(E_{\log}(X),p)$ is a cocycle, then for all $y,z$ we have $$x\bullet y=y\bullet x\quad\text{and}\quad y\bullet(x\bullet z)=(y\bullet x)\bullet z=x\bullet(y\bullet z).$$

(iv). Let $0\leq\alpha\leq1$ be a real number, the product on the Deligne-Beilinson complex $\mathfrak{D}^*(E_{\log}(X),p)$ corresponds to the product on the Dolbeault complex $\mathfrak{D}^*(X,p)$
$$\cup_\alpha:\quad \mathfrak{D}^n(X,p)\otimes \mathfrak{D}^m(X,q)\rightarrow \mathfrak{D}^{n+m}(X,p+q)$$
given by
\begin{align*}
&(a_p,f_p,\omega_p)\cup_\alpha(a_q,f_q,\omega_q)\\
=&(a_p\wedge a_q,f_p\wedge f_q,\alpha(\omega_p\wedge a_q+(-1)^nf_p\wedge\omega_q)+(1-\alpha)(\omega_p\wedge f_q+(-1)^n(a_p\wedge\omega_q))).
\end{align*}
Actually, the map in (iii) is given by $x\bullet y=\psi(\phi(x)\cup_\alpha\phi(y))$ which is independent of the choice of $\alpha$. This product is graded commutative and it is associative up to a natural homotopy.

(v). The natural homotopy mentioned in (iv) is given as a map
$$h_\alpha:\quad \mathfrak{D}^{n}(E_{\log}(X),p)\otimes \mathfrak{D}^{m}(E_{\log}(X),q)\otimes \mathfrak{D}^{l}(E_{\log}(X),r)\to \mathfrak{D}^{n+m+l-1}(E_{\log}(X),p+q+r)$$
by
$$h_\alpha(a\otimes b\otimes c)=\psi(h(\phi(a)\cup_\alpha\phi(b))\cup_\alpha\phi(c))+(-1)^{n+1}\psi(\phi(a)\cup_\alpha h(\phi(b)\cup_\alpha\phi(c)))$$
so that $(a\bullet b)\bullet c-a\bullet(b\bullet c)=h_\alpha d(a\otimes b\otimes c)+d h_\alpha(a\otimes b\otimes c)$.
\end{rem}

The following result will not be used in the present paper.

\begin{lem}\label{lrr}
Let $a\in \mathfrak{D}^{2p}(E_{\log}(X),p)$, $b\in \mathfrak{D}^{2q-1}(E_{\log}(X),q)$ and $c\in \mathfrak{D}^{k}(E_{\log}(X),k)$ with $k\geq1$, then $h_1(a\otimes b\otimes c)=h_1(b\otimes a\otimes c)=0$.
\end{lem}
\begin{proof}
By definition, we have $$h_1(a\otimes b\otimes c)=\psi(h(\phi(a)\cup_1\phi(b))\cup_1\phi(c))-\psi(\phi(a)\cup_1 h(\phi(b)\cup_1\phi(c))).$$
Then we compute
\begin{align*}
&\psi(h(\phi(a)\cup_1\phi(b))\cup_1\phi(c))\\
=&\psi\big(h[(a,a,0)\cup_1(\partial b^{q-1,q-1}-\overline{\partial}b^{q-1,q-1},2\partial b^{q-1,q-1},b)]\cup_1(\partial c^{k-1,0}-\overline{\partial}c^{0,k-1},2\partial c^{k-1,0},c)\big)\\
=&\psi\big(h(a\wedge (\partial b^{q-1,q-1}-\overline{\partial}b^{q-1,q-1}),a\wedge 2\partial b^{q-1,q-1},a\wedge b)\cup_1(\partial c^{k-1,0}-\overline{\partial}c^{0,k-1},2\partial c^{k-1,0},c)\big)\\
=&\psi\big((\pi_{p+q}(\overline{F}^{p+q}(a\wedge b)+\overline{F}^{p+q-1}(a\wedge b)),-2F^{p+q}(\pi_{p+q-1}(a\wedge b)),0)\\
&\cup_1(\partial c^{k-1,0}-\overline{\partial}c^{0,k-1},2\partial c^{k-1,0},c)\big)\\
=&\psi\big((\pi_{p+q}(a\wedge b),0,0)\cup_1(\partial c^{k-1,0}-\overline{\partial}c^{0,k-1},2\partial c^{k-1,0},c)\big)\\
=&\psi\big((\pi_{p+q}(a\wedge b)\wedge (\partial c^{k-1,0}-\overline{\partial}c^{0,k-1}),0,0\big)=0
\end{align*}
and
\begin{align*}
&\psi(\phi(a)\cup_1 h(\phi(b)\cup_1\phi(c)))\\
=&\psi\big(\phi(a)\cup_1 h[(\partial b^{q-1,q-1}-\overline{\partial}b^{q-1,q-1},2\partial b^{q-1,q-1},b)\cup_1 (\partial c^{k-1,0}-\overline{\partial}c^{0,k-1},2\partial c^{k-1,0},c)]\big)\\
=&\psi\big(\phi(a)\cup_1 h((\partial b^{q-1,q-1}-\overline{\partial}b^{q-1,q-1})\wedge (\partial c^{k-1,0}-\overline{\partial}c^{0,k-1}),4\partial b^{q-1,q-1}\wedge \partial c^{k-1,0},\\
&b\wedge (\partial c^{k-1,0}-\overline{\partial}c^{0,k-1})-2\partial b^{q-1,q-1}\wedge c)\big)\\
=&\psi\big((a,a,0)\cup_1 (\pi_{q+k}(b\wedge (\partial c^{k-1,0}-\overline{\partial}c^{0,k-1})-2\partial b^{q-1,q-1}\wedge c),0,0)\big)\\
=&\psi(a\wedge \pi_{q+k}(b\wedge (\partial c^{k-1,0}-\overline{\partial}c^{0,k-1})-2\partial b^{q-1,q-1}\wedge c),0,0)=0
\end{align*}
Therefore, we have $h_1(a\otimes b\otimes c)=0$. The proof of the identity $h_1(b\otimes a\otimes c)=0$ is similar, just notice that $\pi_{p+k}(a\wedge c)=\frac{1}{2}(a\wedge c+(-1)^{p+k}\overline{a}\wedge \overline{c})=\frac{1}{2}(a\wedge c+(-1)^{p+k}(-1)^{p+k-1}a\wedge c)=0$.
\end{proof}

Now we turn to Deligne homology, the homological counterpart of the Deligne-Beilinson cohomology. In general, Deligne homology is defined by means of currents and smooth singular chains (cf. \cite{Ja}). However, in the case we are only interested in, we don't need to use singular chains. We suppose that $X$ is a smooth proper algebraic variety of dimension $d$ over $\C$.

Denote by ${'E}_X^n$ the space of currents of degree $n$ on $X$. By definition, ${'E}_X^n$ is the continuous dual of $E_X^{-n}$ with respect to the Schwartz topology. We may define a differential ${\rm d}: {'E}_X^n\to {'E}_X^{n+1}$ by setting $${\rm d}T(\varphi)=(-1)^nT({\rm d}\varphi)$$ to make ${'E}_X^*$ a complex of vector spaces. The real structure and the bigrading of $E_X^*$ induce a real structure and a bigrading of ${'E}_X^*$.

There are two typical examples of currents that we shall use frequently.

\begin{exa}\label{205}
(i). Let $i: Y\hookrightarrow X$ be a $k$-dimensional subvariety of $X$, there is a $-2k$-degree current $\delta_Y$ introduced by Lelong (cf. \cite{Le}) which is given by $$\delta_Y(\alpha)=\frac{1}{(2\pi i)^k}\int_{Y^{\rm ns}}i^*\alpha,\quad \alpha\in E_X^{2k}$$ where $Y^{\rm ns}$ is the non-singular locus of $Y$. Notice that $\delta_Y$ is actually in ${'E}_{X,\R}^{-2k}(-k)\cap {'E}_X^{-k,-k}$.

(ii). There is a morphism of complexes $E_X^*\to {'E}_X^*[-2d](-d)$ given by $\omega\mapsto[\omega]$, where $[\omega]$ is the current associated to a smooth form which is defined by $$[\omega](\eta)=\frac{1}{(2\pi i)^d}\int_X\eta\wedge\omega.$$
\end{exa}

\begin{defn}\label{206}
Let $X$ be a smooth proper alegbraic variety over $\C$, the real Deligne homology groups of $X$ are
$${'H}_{\mathcal{D}}^*(X,\R(p)):=H^*(\mathfrak{D}({'E}_X,p)).$$ We usually write $$H_n^{\mathcal{D}}(X,\R(p))={'H}_{\mathcal{D}}^{-n}(X,\R(-p)).$$
\end{defn}

The real Deligne-Beilinson cohomology groups and the real Deligne homology groups form a twisted Poincar\'{e} duality theory in the sense of Bloch-Ogus (cf. \cite{BO}).

\begin{thm}\label{207}(Poincar\'{e} Duality)
Let $X$ be a smooth proper alegbraic variety of dimension $d$ over $\C$, then the morphism of complexes $[\cdot]: E_X^*\to {'E}_X^*[-2d](-d)$ induces an isomorphism
$$\xymatrix{H_{\mathcal{D}}^n(X,\R(p))\ar[r]^-{\sim} & {'H}_{\mathcal{D}}^{n-2d}(X,\R(p-d))=H_{2d-n}^{\mathcal{D}}(X,\R(d-p)).}$$
\end{thm}
\begin{proof}
This is \cite[Corollary 5.40]{BKK}.
\end{proof}

\subsection{Equivariant higher Bott-Chern forms}
\label{sec:2.2}
In this subsection, we introduce the higher Bott-Chern forms associated to exact $n$-cubes in the category of equivariant hermitian vector bundles. This generalizes the theory of equivariant Chern-Weil forms to an iterated version, in the sense that degree $n$ forms measure the lack of additivity of degree $n-1$ forms. Hence one obtains a delooping theory of equivariant Chern character for higher K-theory, as Schechtman suggested in \cite{Sch}.
Although the concept of equivariant higher Bott-Chern form is new, its construction is essentially the same as the non-equivariant case. So we refer to Burgos-Wang's original paper \cite{BW} for all details we omit.

Let $<-1,0,1>^n$ be the $n$-th power of the ordered set $<-1,0,1>$, it is naturally partially ordered by the relation $\leq$ with law $(i_1,\ldots,i_n)\leq (j_1,\ldots,j_n)$ if $i_1\leq j_1,\ldots,i_n\leq j_n$. Then $<-1,0,1>^n$ can be viewed as a category whose objects are $(i_1,\ldots,i_n)$ with $i_k\in <-1,0,1>$ and the set of morphisms from $(i_1,\ldots,i_n)$ to $(j_1,\ldots,j_n)$ contains one element if $(i_1,\ldots,i_n)\leq (j_1,\ldots,j_n)$ and is empty otherwise. Now, for an exact category $\mathcal{U}$, an $n$-cube in $\mathcal{U}$ is a functor $\mathcal{F}$ from $<-1,0,1>^n$ to $\mathcal{U}$.

Let $\mathcal{F}$ be an $n$-cube in $\mathcal{U}$, (i). for $1\leq i\leq n$ and $j\in <-1,0,1>$, the $(n-1)$-cube $\partial_i^j\mathcal{F}$ defined by $(\partial_i^j\mathcal{F})_{\alpha_1,\cdots,\alpha_{n-1}}=\mathcal{F}_{\alpha_1,\cdots,\alpha_{i-1},j,\alpha_i,\cdots,\alpha_{n-1}}$ is called a face of $\mathcal{F}$;

(ii). for $\alpha\in <-1,0,1>^{n-1}$ and $1\leq i \leq n$, the $1$-cube $\partial_i^\alpha$ defined by $$\mathcal{F}_{\alpha_1,\cdots,\alpha_{i-1},-1,\alpha_i,\cdots,\alpha_{n-1}}\to \mathcal{F}_{\alpha_1,\cdots,\alpha_{i-1},0,\alpha_i,\cdots,\alpha_{n-1}}\to \mathcal{F}_{\alpha_1,\cdots,\alpha_{i-1},1,\alpha_i,\cdots,\alpha_{n-1}}$$ is called an edge of $\mathcal{F}$.

\begin{defn}\label{208}
An $n$-cube $\mathcal{F}$ in an exact category is said to be exact if all edges of $\mathcal{F}$ are short exact sequences.
\end{defn}

From now on, all cubes we mention are exact, so we just write cubes instead of exact cubes for short. Still let $\mathcal{U}$ be an exact category, we denote by $C_n\mathcal{U}$ the set of all $n$-cubes in $\mathcal{U}$, then $\partial_i^j$ induces a map from $C_n\mathcal{U}$ to $C_{n-1}\mathcal{U}$.

Let $\mathcal{F}$ be an $n$-cube in $\mathcal{U}$. For $1\leq i\leq n+1$, we denote by $S_i^1\mathcal{F}$ the $(n+1)$-cube
$$(S_i^1\mathcal{F})_{\alpha_1,\cdots,\alpha_{n+1}}=\left\{
                                                      \begin{array}{ll}
                                                        0, & \alpha_i=1; \\
                                                        \mathcal{F}_{\alpha_1,\cdots,\alpha_{i-1},\alpha_{i+1},\cdots,\alpha_{n+1}}, & \alpha_i\neq1,
                                                      \end{array}
                                                    \right.$$
such that the morphisms $(S_i^1\mathcal{F})_{\alpha_1,\cdots,\alpha_{i-1},-1,\alpha_{i+1},\cdots,\alpha_{n+1}}\to (S_i^1\mathcal{F})_{\alpha_1,\cdots,\alpha_{i-1},0,\alpha_{i+1},\cdots,\alpha_{n+1}}$ are the identities of $(S_i^1\mathcal{F})_{\alpha_1,\cdots,\alpha_{i-1},\alpha_{i+1},\cdots,\alpha_{n+1}}$. Similarly, we have $S_i^{-1}\mathcal{F}$ and hence maps
$$S_i^j: C_n\mathcal{U}\to C_{n+1}\mathcal{U}\quad\text{for}\quad j=-1,1.$$

The cubes in the image of $S_i^j$ are said to be degenerate. Denote by $\Z C_n\mathcal{U}$ the free abelian group generated by $C_n\mathcal{U}$ and by $D_n$ the subgroup of $\Z C_n{\mathcal{U}}$ generated by all degenerate $n$-cubes. Set $\widetilde{\Z}C_n\mathcal{U}=\Z C_n\mathcal{U}/{D_n}$ and
$$d=\sum_{i=1}^n\sum_{j=-1}^1(-1)^{i+j-1}\partial_i^j: \widetilde{\Z}C_n\mathcal{U}\to \widetilde{\Z}C_{n-1}\mathcal{U}.$$
Then $\widetilde{\Z}C_*\mathcal{U}=(\widetilde{\Z}C_n\mathcal{U},d)$ becomes a homological complex.

Let $X$ be a smooth algebraic variety over $\C$, and let $E$ be a vector bundle on $X$. Recall that a hermitian metric $h$ on $E$ is said to be smooth at infinity if there exists a smooth compactification $\widetilde{X}$ of $X$ and a vector bundle $\widetilde{E}$ on $\widetilde{X}$ with smooth metric $\widetilde{h}$ such that $(\widetilde{E},\widetilde{h})\mid_X$ is isometric to $(E,h)$. As a corollary of \cite[Proposition 2.2]{BW}, any vector bundle $E$ on $X$ admits a smooth at infinity hermitian metric. Now, we regard $X$ as a $\mu_n$-equivariant variety with the trivial action and suppose that $E$ admits a $\mu_n$-structure. Then for any hermitian metric $h$ on $E$, its average with respect to $\mu_n$
$$\vartheta(h):=\frac{1}{n}\sum_{a\in \mu_n}a^*h$$
is a $\mu_n$-invariant hermitian metric. We need the following result.

\begin{prop}\label{209}
Let $X$ be a smooth algebraic variety over $\C$ with trivial $\mu_n$-action, and let $(E,h)$ be a $\mu_n$-equivariant vector bundle on $X$ with smooth at infinity metric. Then the metric $\vartheta(h)$ is still smooth at infinity. Consequently, any $\mu_n$-equivariant vector bundle on $X$ admits a $\mu_n$-invariant smooth at infinity hermitian metric.
\end{prop}
\begin{proof}
Let $(\widetilde{X},\widetilde{E},\widetilde{h})$ be any smooth compactification of $(X,E,h)$. Then the $\mu_n$-structure on $E$ can be naturally extended to $\widetilde{E}$ so that the isomorphism $\widetilde{E}\mid_X\cong E$ is $\mu_n$-equivariant. Since now $\widetilde{E}$ is endowed with a $\mu_n$-structure, the average metric $\vartheta(\widetilde{h})$ for $\widetilde{h}$ makes sense, and it is clear that $\vartheta(\widetilde{h})\mid_X=\vartheta(h)$. Therefore, the metric $\vartheta(h)$ is smooth at infinity.
\end{proof}

In this paper, all hermitian metrics we consider are smooth at infinity metrics. For a smooth algebraic variety $X/\C$ with a $\mu_n$-action, we shall  denote by $\widehat{\mathcal{P}}(X)$ the exact category of $\mu_n$-equivariant vector bundles on $X$ equipped with smooth at infinity hermitian metrics, and by $\widehat{\mathcal{P}}(X,\mu_n)$ the exact category of $\mu_n$-equivariant vector bundles on $X$ equipped with $\mu_n$-invariant smooth at infinity hermitian metrics. Define $R_n=\R$ if $n=1$ and $R_n=\C$ otherwise. In the rest of this subsection, we focus on the cubes in the category $\mathcal{U}=\widehat{\mathcal{P}}(X,\mu_n)$ whence the $\mu_n$-action on $X$ is trivial, which are called the equivariant hermitian cubes. In particular, we shall construct a morphism of complexes from $\widetilde{\Z}C^*\mathcal{U}$ to $\bigoplus_{p\geq0}\widetilde{D}^*(X,p)[2p]_{R_n}$, called the equivariant higher Bott-Chern form, where $\widetilde{D}^*(X,p)$ is a complex computing the real Deligne-Beilinson cohomology of $X$ and $\widetilde{\Z}C^*\mathcal{U}$ is the complex associated to the homological complex $\widetilde{\Z}C_*\mathcal{U}$ defined by $\widetilde{\Z}C^k\mathcal{U}=\widetilde{\Z}C_{-k}\mathcal{U}$.

Notice that the cartesian product of projective lines $(\mathbb{P}^1)^\cdot$ over any base has a cocubical scheme structure, the coface and codegeneracy maps
$$d_j^i: (\mathbb{P}^1)^k\to (\mathbb{P}^1)^{k+1},\quad i=1,\cdots,k,j=0,1,$$
$$s^i: (\mathbb{P}^1)^k\to (\mathbb{P}^1)^{k-1},\quad i=1,\cdots,k,$$
are given by
$$d_0^i(x_1,\cdots,x_k)=(x_1,\cdots,x_{i-1},(0:1),x_i,\cdots,x_k),$$
$$d_1^i(x_1,\cdots,x_k)=(x_1,\cdots,x_{i-1},(1:0),x_i,\cdots,x_k),$$
$$s^i(x_1,\cdots,x_k)=(x_1,\cdots,x_{i-1},x_{i+1},\cdots,x_k).$$

With the above notations, the complexes $\mathfrak{D}^*(X\times(\mathbb{P}^1)^\cdot,p)$ form a cubical complex with face and degeneracy maps
$$d_i^j=({\rm Id}\times d_j^i)^*\quad\text{and}\quad s_i=({\rm Id}\times s^i)^*.$$

We shall write $\mathfrak{D}_{\mathbb{P}}^{r,k}(X,p)=\mathfrak{D}^r(X\times (\mathbb{P}^1)^{-k},p)$ and denote by $\mathfrak{D}_{\mathbb{P}}^{*,*}(X,p)$ the associated double complex with differentials
$$d'=d_{\mathfrak{D}}\quad\text{and}\quad d''=\sum(-1)^{i+j-1}d_i^j.$$
Let us write $p_0: X\times(\mathbb{P}^1)^k\to X$ for the projection over the first factor and $p_i: X\times(\mathbb{P}^1)^k\to \mathbb{P}^1, i=1,\cdots,k$, for the projection over the $i$-th projective line. Denote by $(x:y)$ the homogeneous coordinates of $\mathbb{P}^1$. The real function $$g=\log\frac{x\overline{x}+y\overline{y}}{x\overline{x}}$$ is defined on the open set $\mathbb{P}^1\setminus\{0\}$ with logarithmic singularity at $0$. Let $\omega=\partial\overline{\partial}g\in (2\pi i)E_{\mathbb{P}^1,\R}^2$, then $\omega$ is a K\"{a}hler form over $\mathbb{P}^1$. Set $\omega_i=p_i^*\omega\in E_{\log}^*(X\times(\mathbb{P}^1)^k).$ For an element $$x=(\alpha,\beta,\eta)\in \mathfrak{D}^r(X\times(\mathbb{P}^1)^k,p),$$ we shall write $$\omega_i\wedge x=(\omega_i\wedge\alpha,\omega_i\wedge\beta,\omega_i\wedge\eta)\in \mathfrak{D}^{r+2}(X\times(\mathbb{P}^1)^k,p+1).$$

\begin{defn}\label{210}
Let $\widetilde{\mathfrak{D}}^{*,*}(X,p)$ be the double complex given by $$\widetilde{\mathfrak{D}}^{r,k}(X,p)=\mathfrak{D}_{\mathbb{P}}^{r,k}(X,p)/{\sum_{i=1}^{-k}s_i\big(\mathfrak{D}_{\mathbb{P}}^{r,k+1}(X,p)\big)\oplus
\omega_i\wedge s_i\big(\mathfrak{D}_{\mathbb{P}}^{r-2,k+1}(X,p-1)\big)}.$$
We define $\widetilde{\mathfrak{D}}^*(X,p)$ as the associated simple complex. The differential of this complex will be denoted by $d$. Notice that this complex is meant to kill the degenerate classes and the classes coming from the projective spaces.
\end{defn}

An interesting fact for the complex $\widetilde{\mathfrak{D}}^*(X,p)$ is the following result whose proof is a repetition of the proofs of \cite[Proposition 1.2 and Lemma 1.3]{BW}.

\begin{prop}\label{211}
The natural morphism of complexes $$\iota: \mathfrak{D}^*(X,p)=\widetilde{\mathfrak{D}}^{*,0}(X,p)\to \widetilde{\mathfrak{D}}^*(X,p)$$ is a quasi-isomorphism.
\end{prop}

\begin{rem}\label{212}
Replace $\mathfrak{D}_{\mathbb{P}}^{r,k}(X,p)$ by $\mathfrak{D}_{\mathbb{P}}^{r,k}(E_{\log}(X),p):=\mathfrak{D}^r\big(E_{\log}(X\times(\mathbb{P}^1)^{-k}),p\big)$, we have analogue $\widetilde{\mathfrak{D}}^*(E_{\log}(X),p)$ for the Deligne complex $\mathfrak{D}^*(E_{\log}(X),p)$ and a quasi-isomorphism
$$\iota: \mathfrak{D}^*(E_{\log}(X),p)=\widetilde{\mathfrak{D}}^{*,0}(E_{\log}(X),p)\to \widetilde{\mathfrak{D}}^*(E_{\log}(X),p).$$
\end{rem}

\begin{defn}\label{213}
Set $\widetilde{D}^*(X,p)=\widetilde{\mathfrak{D}}^*(E_{\log}(X),p)$, it is a complex of real vector spaces whose cohomology groups are isomorphic to the real Deligne-Beilinson cohomology groups of $X$.
\end{defn}

When $X$ is proper, Burgos and Wang gave in \cite[Section 6]{BW} a quasi-inverse of the quasi-isomorphism $\iota: D^*(X,p)\to \widetilde{D}^*(X,p)$ by means of the following Wang forms
$$W_n=\frac{(-1)^n}{2n!}\sum_{i=1}^n(-1)^iS_n^i,$$ with $$S_n^i=\sum_{\sigma\in \mathfrak{S}_n}(-1)^\sigma\log\mid z_{\sigma(1)}\mid^2\frac{dz_{\sigma(2)}}{z_{\sigma(2)}}\wedge\cdots\wedge\frac{dz_{\sigma(i)}}{z_{\sigma(i)}}\wedge\frac{d\overline{z}_{\sigma(i+1)}}{\overline{z}_{\sigma(i+1)}}
\wedge\cdots\wedge\frac{d\overline{z}_{\sigma(n)}}{\overline{z}_{\sigma(n)}}$$
and $\mathfrak{S}_n$ stands for the $n$-th symmetric group.

Actually, for any $\alpha\in \mathfrak{D}^r(X\times (\mathbb{P}^1)^n,p)$, the expression
$${p_0}_*(\alpha\bullet W_n)=\left\{
                               \begin{array}{ll}
                                 \frac{1}{(2\pi i)^n}\int_{(\mathbb{P}^1)^n}\alpha\bullet W_n, & n>0; \\
                                 \alpha, & n=0.
                               \end{array}
                             \right.$$
provides a quasi-isomorphism of complexes $\varphi: \widetilde{D}^*(X,p)\to D^*(X,p)$ which is the quasi-inverse of $\iota$.

Now, let $\overline{E}$ be a hermitian $k$-cube in the category $\mathcal{U}=\widehat{\mathcal{P}}(X,\mu_n)$. We firstly assume that $\overline{E}$ is an emi-cube, that means the metrics on the quotient terms in all edges of $\overline{E}$ are induced by the metrics on the middle terms (cf. \cite[Definition 3.5]{BW}). In \cite[(3.7)]{BW}, Burgos and Wang associated to $\overline{E}$ a hermitian locally free sheaf ${\rm tr}_k(\overline{E})$ on $X\times (\mathbb{P}^1)^k$ in an inductive way. This ${\rm tr}_k(\overline{E})$ is called the $k$-transgression bundle of $\overline{E}$, let us recall its construction. If $k=1$, as an emi-$1$-cube, $\overline{E}$ is a short exact sequence
$$\xymatrix{ 0\ar[r] & \overline{E}_{-1}\ar[r]^-{i} & \overline{E}_0\ar[r] & \overline{E}_1 \ar[r] & 0},$$
where the metric of $\overline{E}_1$ is induced by the metric of $\overline{E}_0$. Then ${\rm tr}_1(\overline{E})$ is the cokernel with quotient metric of the map $E_{-1}\to E_{-1}\otimes\mathcal{O}(1)\oplus E_0\otimes\mathcal{O}(1)$ by the rule
$e_{-1}\mapsto e_{-1}\otimes \sigma_{\infty}\oplus i(e_{-1})\otimes \sigma_0$.
Here $\sigma_0$ (resp. $\sigma_{\infty}$) is the canonical section of the tautological bundle $\mathcal{O}(1)$ on $\mathbb{P}^1$ which vanishes only at $0$ (resp. $\infty$), and $\mathcal{O}(1)$ is endowed with the Fubini-Study metric. If $k>1$, suppose that transgression bundle is defined for $k-1$. Let ${\rm tr}_1(\overline{E})$ be the emi-$(k-1)$-cube over $X\times \mathbb{P}^1$ given by ${\rm tr}_1(\overline{E})_\alpha={\rm tr}_1(\partial^\alpha_1(\overline{E}))$, then ${\rm tr}_k(\overline{E})$ is defined as ${\rm tr}_{k-1}\big({\rm tr}_1(\overline{E})\big)$.

On the other hand, for any hermitian cube $\overline{E}$ in the category $\mathcal{U}=\widehat{\mathcal{P}}(X,\mu_n)$, there is a unique way to change the metrics on $E_\alpha$ for $\alpha\nleq0$ such that the obtained new hermitian cube is emi (cf. \cite[Proposition 3.6]{BW}). Actually, for $i=1,\ldots,k$, define $\lambda_i^1\overline{E}$ to be $$(\lambda_i^1\overline{E})_\alpha=\left\{
                                                                                          \begin{array}{ll}
                                                                                            (E_\alpha,h_\alpha), & \text{if } \alpha_i=-1,0; \\
                                                                                            (E_\alpha,h'_\alpha), & \text{if } \alpha_i=1,
                                                                                          \end{array}
                                                                                        \right.$$
where $h'_\alpha$ is the metric induced by $h_{\alpha_1,\ldots,\alpha_{i-1},0,\alpha_{i+1},\ldots,\alpha_k}$. Thus $\lambda_i^1\overline{E}$ has the same locally free sheaves as $\overline{E}$, but the metrics on the face $\partial_i^1E$ are induced by the metrics of the face $\partial_i^0\overline{E}$.
To define the correct higher Bott-Chern forms, we have to measure the difference between $\overline{E}$ and $\lambda_i^1\overline{E}$. Let $\lambda_i^2(\overline{E})$ be the hermitian $k$-cube determined by $\partial_i^{-1}\lambda_i^2(\overline{E})=\partial_i^1\overline{E}$, $\partial_i^0\lambda_i^2(\overline{E})=\partial_i^1\lambda_i^1(\overline{E})$, and $\partial_i^1\lambda_i^2(\overline{E})=0$. Set $\lambda_i=\lambda_i^1+\lambda_i^2$, $\lambda=\lambda_k\circ\cdots\circ\lambda_1$ if $k\geq1$ and $\lambda={\rm Id}$ otherwise.

\begin{prop}\label{214}
Write $\widetilde{\Z}C_*(X,\mu_n):=\widetilde{\Z}C_*\mathcal{U}$ and denote by $\widetilde{\Z}C_*^{\rm emi}(X,\mu_n)$ the subcomplex of $\widetilde{\Z}C_*(X,\mu_n)$ generated by emi-cubes. Then the map $\lambda$ induces a morphism of complexes $$\widetilde{\Z}C_*(X,\mu_n)\to \widetilde{\Z}C_*^{\rm emi}(X,\mu_n)$$ which is the quasi-inverse of the inclusion $\widetilde{\Z}C_*^{\rm emi}(X,\mu_n)\hookrightarrow \widetilde{\Z}C_*(X,\mu_n)$.
\end{prop}

\begin{defn}\label{215}
Let $\overline{E}$ be a hermitian $k$-cube in the category $\mathcal{U}=\widehat{\mathcal{P}}(X,\mu_n)$, the equivariant higher Bott-Chern form associated to $\overline{E}$ is defined as $${\rm ch}_g^k(\overline{E}):={\rm ch}_g^0\big({\rm tr}_k\circ \lambda(\overline{E})\big),$$ where ${\rm ch}_g^0(F,h_F)=\sum_{l=1}^n\zeta_l{\rm Tr}(\exp(-K_l))$ is the equivariant Chern-Weil form associated to an equivariant hermitian vector bundle
$\overline{F}=\oplus_{l=1}^n\overline{F}_l$ with curvature form $K_l$ for $\overline{F}_l$.
\end{defn}

\begin{thm}\label{216}
The equivariant higher Bott-Chern forms induce a morphism of complexes
$$\xymatrix{ \widetilde{\Z}C^*(X,\mu_n) \ar[rr]^-{\lambda} && \widetilde{\Z}C^*_{\rm emi}(X,\mu_n) \ar[rr]^-{{\rm ch}_g^0({\rm tr}_*(\cdot))} && \bigoplus_{p\geq0}\widetilde{D}^*(X,p)[2p]_{R_n},}$$ which is denoted by ${\rm ch}_g$.
\end{thm}

\subsection{Higher K-theory and equivariant regulator maps}
\label{sec:2.3}
Let $X$ be a $\mu_n$-equivariant arithmetic scheme over an arithmetic ring $(D,\Sigma,F_\infty)$, on which the $\mu_n$-action is not necessarily trivial. In this subsection, we give the definition of the equivariant higher arithmetic K-groups $\widehat{K}_m(X,\mu_n)$. As the non-equivariant case, we shall define them to be the homotopy groups of the homotopy fibres of regulator maps.

Firstly, denote $X_\R:=(X(\C),F_\infty)$ the real variety associated to $X$ where $X(\C)$ is the analytification of $X_\C$ which is an algebraic complex manifold and $F_\infty$ is the antiholomorphic involution of $X(\C)$ induced by the conjugate-linear involution $F_\infty$ over $(D,\Sigma,F_\infty)$. For any sheaf of complex vector spaces $V$ with a real structure over $X_\R$, we denote by $\sigma$ the involution given by $$\omega\mapsto \overline{F_\infty^*(\omega)}.$$ The same notation goes to any subsheaf of abelian groups of $V$, which is invariant under complex conjugation. According to \cite[Section 5.3]{BKK}, there is a complex of (fine) sheaves whose groups of global sections equal $D^*(X(\C),p)$. Notice that the involution $\sigma$ respects the real structure and the Hodge filtration, thus $\sigma$ induces an involution over the complex $D^*(X(\C),p)$ and hence over $\widetilde{D}^*(X(\C),p)$, which is still denoted by $\sigma$. Write $D^*(X_\R,p):=D^*(X(\C),p)^\sigma$ for the subcomplex of $D^*(X(\C),p)$ consisting of the fixed elements under $\sigma$, we define the real Deligne-Beilinson cohomology of $X$ as
$$H_D^*(X,p):=H^*(D^*(X_\R,p)).$$ Since we are working with complex of vector spaces, the functor $(\cdot)^\sigma$ is exact. Then $H_D^*(X,p)$ is actually $H_D^*(X(\C),p)^\sigma$. An analogous definition can be given for the Deligne homology of $X$ whence $X$ is proper over $D$.

Now, recall that $\mathcal{P}(X,\mu_n)$ denotes the exact category of $\mu_n$-equivariant vector bundles on $X$. Then $\mathcal{P}(X,\mu_n)$ can be viewed as a Waldhausen category with cofibration sequences are short exact sequences of vector bundles and weak equivalences are isomorphisms of vector bundles. The $S$-construction of $\mathcal{P}(X,\mu_n)$ is a simplicial category with cofibrations and weak equivalences
\begin{align*}
  {\rm w}S.\mathcal{P}(X,\mu_n): \triangle^{\rm op} & \longrightarrow ({\rm cat}) \\
  [n] & \longmapsto {\rm w}S_n\mathcal{P}(X,\mu_n).
\end{align*}
Here ${\rm w}S_n\mathcal{P}(X,\mu_n)$ is a category whose objects are functors
\begin{align*}
  A: {\rm Ar}[n] & \longrightarrow \mathcal{P}(X,\mu_n) \\
  (i,j) & \longmapsto A_{i,j}
\end{align*}
satisfying the property that for every $0\leq i\leq n$, $A_{i,i}=0$ and that for every triple $i\leq j\leq k$,
$$0\to A_{i,j}\to A_{i,k}\to A_{j,k}\to 0$$ is a short exact sequence. Notice that the set ${\rm Ob}({\rm w}S_n\mathcal{P}(X,\mu_n))$ can be identified with the set of injections $$A_{0,1}\rightarrowtail A_{0,2}\rightarrowtail \cdots\rightarrowtail A_{0,n}$$ with quotients $A_{i,j}\simeq A_{0,j}/{A_{0,i}}$ for each $i<j$. The morphisms in the category ${\rm w}S_n\mathcal{P}(X,\mu_n)$ are natural transformations between functors $A\to A'$ such that $A_{i,j}\to A'_{i,j}$ is isomorphism for every pair $i\leq j$.

\begin{defn}\label{217}
The equivariant Waldhausen K-theory space of $X$ is the geometric realization of the bisimplicial set $N.{\rm w}S.\mathcal{P}(X,\mu_n)$ with base point $0\in S_0$, where $N.$ stands for the nerve of a small category.
\end{defn}

The rule $[n]\mapsto {\rm Ob}({\rm w}S_n\mathcal{P}(X,\mu_n))$ gives us another simplicial set, which we shall denote by $S(X,\mu_n)$. We have the following results.

\begin{thm}\label{218}
There are homotopy equivalence maps between geometric realizations of simplicial sets with base points
$$\mid S(X,\mu_n)\mid\simeq \mid N.{\rm w}S.\mathcal{P}(X,\mu_n)\mid\simeq BQ\mathcal{P}(X,\mu_n).$$ Therefore for $m\geq0$, we have
$$\pi_{m+1}(\mid S(X,\mu_n)\mid,0)\cong K_m(X,\mu_n).$$
\end{thm}
\begin{proof}
The existence of the second homotopy equivalence is a classic result given by Waldhausen in \cite[Section 1.9]{Wal}. A proof of the existence of the first homotopy equivalence can be found in \cite[(1.2)]{Sch}.
\end{proof}

Let us denote by $\widehat{\mathcal{P}}(X)$ the exact category of $\mu_n$-equivariant vector bundles on $X$ with smooth at infinity metrics (not necessarily $\mu_n$-invariant), and by $\widehat{S}(X)$ the simplicial set associated to the Waldhausen $S$-construction of $\widehat{\mathcal{P}}(X)$ as above. Observe that the forgetful functor (forget about the metrics) $\pi: \widehat{\mathcal{P}}(X)\to \mathcal{P}(X,\mu_n)$ induces an equivalence of categories, so we have $$\mid \widehat{S}(X)\mid\simeq \mid S(X,\mu_n)\mid$$ and
$$K_m(X,\mu_n)\cong \pi_{m+1}(\mid \widehat{S}(X)\mid,0)$$ for any $m\geq 0$. Similarly, if the $\mu_n$-action on $X$ is trivial, we denote by $\widehat{S}(X,\mu_n)$ the Waldhausen $S$-construction of $\widehat{\mathcal{P}}(X,\mu_n)$, which is the exact category of $\mu_n$-equivariant vector bundles on $X$ with $\mu_n$-invariant smooth at infinity metrics, then by Proposition~\ref{209} we have homotopy equivalence $\mid \widehat{S}(X,\mu_n)\mid\simeq \mid S(X,\mu_n)\mid$. In particular, there is a homotopy equivalence map $$\vartheta:\quad \mid \widehat{S}(X)\mid\simeq \mid \widehat{S}(X,\mu_n)\mid.$$
This map $\vartheta$ is realized by making the $\mu_n$-average of a smooth at infinity metric.

Next, to give the simplicial description of the equivariant regulator maps, let us associated, to each element in $S_k\widehat{\mathcal{P}}(X,\mu_n)$, a hermitian $(k-1)$-cube, the same as Burgos and Wang did for the non-equivariant case. Firstly, for $k=1$, we write
$${\rm Cub}(A_{0,1})=A_{0,1}.$$ Suppose that the map ${\rm Cub}$ is defined for all $l<k$. Let $A\in S_k\widehat{\mathcal{P}}(X,\mu_n)$, then ${\rm Cub}A$ is the $(k-1)$-cube with
\begin{align*}
\partial_1^{-1}{\rm Cub}A & =s_{k-2}^1\cdots s_1^1(A_{0,1}), \\
\partial_1^{0}{\rm Cub}A & ={\rm Cub}(\partial_1A), \\
\partial_1^{1}{\rm Cub}A & ={\rm Cub}(\partial_0A).
\end{align*}

Let $\Z\widehat{S}_*(X,\mu_n)$ be the simplicial abelian group generated by simplicial set $\widehat{S}(X,\mu_n)$. That is, $\Z\widehat{S}_k(X,\mu_n)$ is the free abelian group generated by $S_k\widehat{\mathcal{P}}(X,\mu_n)$. We shall write $\mathcal{N}(\Z\widehat{S}_*(X,\mu_n))$ for the Moore complex associated to $\Z\widehat{S}_*(X,\mu_n)$, this is a homological complex with differential $d: \Z\widehat{S}_k(X,\mu_n)\to \Z\widehat{S}_{k-1}(X,\mu_n)$ given by
$$d=\sum_{i=0}^k(-1)^i\partial_i.$$

\begin{prop}\label{219}
The map ${\rm Cub}$ defined above extends by linearity to a morphism of homological complexes
$${\rm Cub}:\quad \mathcal{N}(\Z\widehat{S}_*(X,\mu_n))\to \widetilde{\Z}C_*(X,\mu_n)[-1].$$
\end{prop}
\begin{proof}
This is \cite[Corollary 4.8]{BW}.
\end{proof}

\begin{rem}\label{220}
The Moore complex $\mathcal{N}(\cdot)$ associated to a simplicial abelian group actually gives an equivalence between the category of simplicial abelian groups $s{\rm Ab}$ and the category of homological complexes which are zero in negative degrees $\mathcal{I}_{+}$. Its inverse is given by the Dold-Puppe functor $\mathcal{K}: \mathcal{I}_+\to s{\rm Ab}$. Hence, we have a simplicial map $${\rm Cub}:\quad \Z\widehat{S}_*(X,\mu_n)\to \mathcal{K}(\widetilde{\Z}C_*(X,\mu_n)[-1]).$$
\end{rem}

Let $X$ be a $\mu_n$-equivariant arithmetic scheme over an arithmetic ring $(D,\Sigma,F_\infty)$, and let $X_{\mu_n}$ be the fixed point subscheme. It's a general theorem which states that if $X$ is regular, then $X_{\mu_n}$ is also regular. (cf. for instance \cite[Proposition 3.1]{Th2}) So, the generic fibre of $X_{\mu_n}$ is also smooth and hence $X_{\mu_n}$ is a $\mu_n$-equivariant arithmetic scheme.

\begin{defn}\label{221}
Let $X$ and $X_{\mu_n}$ be as above, and let $i: X_{\mu_n}\hookrightarrow X$ be the natural inclusion. Denote by $\widetilde{D}^{2p-*}(X_{\mu_n},p)$ the homological complex associated to the complex $\tau_{\leq0}\big(\widetilde{D}^*(X_{\mu_n},p)[2p]\big)$ which is the canonical truncation of $\widetilde{D}^*(X_{\mu_n},p)[2p]$ at degree $0$. We define two simplicial maps $$\xymatrix{\varphi: \widehat{S}(X) \ar[r]^-{i^*} & \widehat{S}(X_{\mu_n})}$$ and $$\xymatrix{\phi: \widehat{S}(X_{\mu_n},\mu_n) \ar[r]^-{\rm Hu} & \Z\widehat{S}(X_{\mu_n},\mu_n) \ar[d]^-{\rm Cub} & \\ & \mathcal{K}(\widetilde{\Z}C_*(X_{\mu_n},\mu_n)[-1]) \ar[r]^-{\mathcal{K}({\rm ch}_g)} & \mathcal{K}(\bigoplus_{p\geq0}\widetilde{D}^{2p-*}(X_{\mu_n},p)[-1]_{R_n}),}$$
where $i^*$ is the pull-back map induced by $i$ and ${\rm Hu}$ is the Hurewicz map. The continuous maps induced by $\varphi$ and $\phi$ between corresponding geometric realizations of simplicial sets are still denoted by $\varphi$ and $\phi$.
\end{defn}

\begin{defn}\label{222}
Recall that $\vartheta:\quad \mid \widehat{S}(X_{\mu_n})\mid\simeq \mid \widehat{S}(X_{\mu_n},\mu_n)\mid$ denotes the homotopy equivalence which is the geometric realization of a simplicial map
$$\vartheta: \widehat{S}(X_{\mu_n}) \to \widehat{S}(X_{\mu_n},\mu_n),$$
by making the $\mu_n$-average of a hermitian metric. We define $$\widetilde{{\rm ch}_g}:\quad \widehat{S}(X) \to  \mathcal{K}(\bigoplus_{p\geq0}\widetilde{D}^{2p-*}(X_{\mu_n},p)[-1]_{R_n})$$ as the composition $\phi\circ\vartheta\circ\varphi$.
\end{defn}

\begin{rem}\label{newformalism}
It is clear that the simplicial map $\widetilde{{\rm ch}_g}$ can be also described as the composition of simplicial maps
$$\xymatrix{\widehat{S}(X) \ar[r]^-{\rm Hu} & \Z\widehat{S}(X) \ar[d]^-{\rm Cub} & \\ & \mathcal{K}(\widetilde{\Z}C_*(X)[-1]) \ar[r]^-{\vartheta\circ i^*} & \mathcal{K}(\widetilde{\Z}C_*(X_{\mu_n},\mu_n)[-1])}$$
followed by
$$\xymatrix{\mathcal{K}(\widetilde{\Z}C_*(X_{\mu_n},\mu_n)[-1]) \ar[r]^-{\mathcal{K}({\rm ch}_g)} & \mathcal{K}(\bigoplus_{p\geq0}\widetilde{D}^{2p-*}(X_{\mu_n},p)[-1]_{R_n}).}$$
\end{rem}

Next, let us give the definitions of higher equivariant arithmetic K-groups and equivariant regulator maps.

\begin{defn}\label{223}
Let $X$ be a $\mu_n$-equivariant arithmetic scheme over an arithmetic ring $(D,\Sigma,F_\infty)$, and let $X_{\mu_n}$ be the fixed point subscheme. The higher equivariant arithmetic K-groups of $X$ are defined as $$\widehat{K}_m(X,\mu_n):=\pi_{m+1}\big(\text{homotopy fibre of }\mid\widetilde{{\rm ch}_g}\mid\big)\quad\text{for}\quad m\geq1,$$ and the equivariant regulator maps $${\rm ch}_g:\quad K_m(X,\mu_n)\to \bigoplus_{p\geq0}H_D^{2p-m}(X_{\mu_n},\R(p))_{R_n}$$ are defined as the homomorphisms induced by $\widetilde{{\rm ch}_g}$ on the level of homotopy groups.
\end{defn}

Whence $n=1$, $\vartheta$ is the identity map, then Definition~\ref{223} recovers the results given in \cite{BW} for the non-equivariant case, namely
$${\rm ch}_g:\quad K_m(X,\mu_1)\to \bigoplus_{p\geq0}H_D^{2p-m}(X,\R(p))$$
equals the Beilinson's regulator map.

Finally, by the definition of the higher equivariant arithmetic K-groups, we have the long exact sequence
$$\cdots \to \widehat{K}_m(X,\mu_n) \to K_m(X,\mu_n) \to \bigoplus_{p\geq0}H_{\mathcal{D}}^{2p-m}(X_{\mu_n},\R(p))_{R_n} \to \widehat{K}_{m-1}(X,\mu_n) \to \cdots$$ ending with
$$\xymatrix{\cdots\to K_1(X,\mu_n) \ar[r] & \bigoplus_{p\geq0}H_{\mathcal{D}}^{2p-1}(X_{\mu_n},\R(p))_{R_n}
\ar[d] & \\
& \pi_1\big(\text{homotopy fibre of }\widetilde{{\rm ch}_g}\big) \ar[r] & K_0(X,\mu_n)\to \bigoplus_{p\geq0}H_{\mathcal{D}}^{2p}(X_{\mu_n},\R(p))_{R_n}.}$$

\begin{rem}\label{224}
If $X$ is proper, any $\mu_n$-invariant metric on vector bundle over $X$ is automatically smooth at infinity, then we may replace $\widehat{S}(X)$ by $\widehat{S}(X,\mu_n)$ in the definition of the higher equivariant arithmetic K-groups. Also in this case, we may use the complex $D^*(X_{\mu_n},p)$ instead of $\widetilde{D}^*(X_{\mu_n},p)$. Furthermore, replacing the canonical truncation $\tau_{\leq0}\big(D^*(X_{\mu_n},p)[2p]\big)$ by the b\^{e}te truncation $\sigma_{<0}\big(D^*(X_{\mu_n},p)[2p]\big)$, we may extend Definition~\ref{223} to $m=0$ and we have corresponding long exact sequence ending with
$$\cdots \to K_1(X,\mu_n) \to \bigoplus_{p\geq0}\big(D^{2p-1}(X_{\mu_n},\R(p))/{({\rm Im}\partial+{\rm Im}\overline{\partial})}\big)_{R_n}
\to \widehat{K}_0(X,\mu_n) \to K_0(X,\mu_n)\to 0.$$
Here the K-group $\widehat{K}_0(X,\mu_n)$ agrees (up to a normalization) with the equivariant arithmetic K$_0$-group defined in \cite{KR} in the spirit of Gillet-Soul\'{e}, and $\pi_1\big(\text{homotopy fibre of }\widetilde{{\rm ch}_g}\big)$ is a subgroup of $\widehat{K}_0(X,\mu_n)$. An easy diagram chasing argument shows that the cokernel $\widehat{K}_0(X,\mu_n)/{\pi_1}$ fits in the following exact sequence
$$0\to \bigoplus_{p\geq0}\big(D^{2p-1}(X_{\mu_n},\R(p))/{{\rm Ker}(\partial\overline{\partial})}\big)_{R_n}\to \widehat{K}_0(X,\mu_n)/{\pi_1}\to K_0(X,\mu_n)/{{\rm Ker}({\rm ch_g^0})}\to 0,$$
where ${\rm ch_g^0}: K_0(X,\mu_n)\to \bigoplus_{p\geq0}H_{\mathcal{D}}^{2p}(X_{\mu_n},\R(p))_{R_n}$ is the equivariant Chern character map.
\end{rem}

The last result we would like to mention is the following.

\begin{prop}\label{torsion}
Let $s({\rm ch_g})$ denote the simple complex associated to the chain morphism
$$\xymatrix{ {\rm ch_g}: & \widetilde{\Z}C_*(X) \ar[r]^-{\vartheta\circ i^*} & \widetilde{\Z}C_*(X_{\mu_n},\mu_n) \ar[r]^-{{\rm ch_g}} &
\bigoplus_{p\geq0}\widetilde{D}^{2p-*}(X_{\mu_n},p)_{R_n}.}$$
Then, for any $m\geq1$, there is an isomorphism $$\widehat{K}_m(X,\mu_n)_\Q\cong H_m(s({\rm ch_g}),\Q).$$
\end{prop}
\begin{proof}
This follows from McCarthy's result \cite{MC} together with \cite[Lemma 2.1]{Fe}.
\end{proof}

\section{Arithmetic K-theory with supports}
\label{sec:3}
Let $X$ be a scheme, and let $Y$ be a closed subscheme of $X$ with complement $j: U\hookrightarrow X$. The algebraic K-groups of $X$ with supports in $Y$ are defined as $$K_{Y,m}(X):=\pi_{m+1}\big(\text{homotopy fibre of }j^*: BQ(\mathcal{P}(X))\to BQ(\mathcal{P}(U))\big),$$ so that one may have a long exact sequence
$$\cdots\to K_{Y,m}(X)\to K_m(X)\to K_m(U)\to K_{Y,m-1}(X)\to\cdots.$$
In this section, we shall develop the arithmetic and the equivariant analogue of the above algebraic K-groups with supports. Since we work with simplicial sets and their geometric realizations, we shall focus on the category $\mathcal{V}$ of compactly generated spaces with nondegenerate base points. The following theorem will be frequently used in the rest of this paper.

\begin{thm}\label{301}
Let $f: X\to Y$ be a map in $\mathcal{V}$, we denote by $Ff$ the homotopy fibre associated to $f$. Recall that $Ff:=X\times_Y PY$, the fibre product with respect to $f$ and the end-point evaluation $PY\to Y$ from the based path space of $Y$ to $Y$.

(1). Consider the following diagram, in which the right square commutes up to homotopy and the rows are canonical fibre sequence
$$\xymatrix{ \Omega Y \ar[r]^-{\iota} \ar[d]_-{\Omega\alpha} & Ff \ar[r]^-{\pi} \ar@{.>}[d]^-{\gamma} & X \ar[r]^-f \ar[d]^-\beta & Y \ar[d]^-\alpha \\
\Omega Y' \ar[r]_-{\iota} & Ff' \ar[r]^-{\pi} & X' \ar[r]_-{f'} & Y'.}$$
There exists a map $\gamma$ such that the middle square commutes and the left square commutes up to homotopy. If the right square commutes strictly, then there is a unique $\gamma=F(\alpha,\beta)$ such that both left squares commute;

(2). Let $f$ be homotopic to $h\circ g$ in the following braid of fibre sequences and let $j''=\iota(g)\circ \Omega\pi(h)$. There are maps $j$ and $j'$ such that the diagram commutes up to homotopy, and there is a homotopy equivalence $\xi: Fj\to Fg$ such that $j'\circ \xi\simeq \pi(j)$ and $j''\simeq\xi\circ\iota(j)$
$$\xymatrix{ \Omega Fh \ar[rr]^-{j''} \ar[dr]_-{\Omega\pi(h)} && Fg \ar[rr]^-{\pi(g)} \ar@{.>}[dr]^-{j'} && Z \ar[rr]^-f \ar[dr]^-g && X \\
& \Omega Y \ar[ur]^-{\iota(g)} \ar[dr]_-{\Omega h} && Ff \ar[ur]^-{\pi(f)} \ar@{.>}[dr]^-{j} && Y \ar[ur]_-{h} & \\
&& \Omega X \ar[ur]^-{\iota(f)} \ar[rr]_-{\iota{h}} && Fh \ar[ur]_-{\pi(h)} && .}$$
If $f=h\circ g$, then there are unique maps $j$ and $j'$ such that the above diagram commutes, and then $\xi$ can be so chosen that $j'\circ \xi=\pi(j)$ and $j''=\xi\circ\iota(j)$.

(3). If the maps $\beta: X\to X'$ and $\alpha: Y\to Y'$ in (1) are both weak equivalences, then the induced map $\gamma: Ff\to Ff'$ is also a weak equivalence. If furthermore, $X,X',Y$ and $Y'$ are all CW-complexes, then $\gamma: Ff\to Ff'$ is a homotopy equivalence.
\end{thm}
\begin{proof}
The first part (1) is classic and is summarized in \cite[Lemma 1.2.3]{MP}, which is called the fill-in lemma. The second part (2) is
\cite[Lemma 1.2.7]{MP}, which is called the Verdier lemma. The existences of $j$ and $j'$ are guaranteed by (1). These two lemmas actually describe the homotopy category ${\rm Ho}\mathcal{V}$ as a ``pretriangulated category". The first result in (3) follows from the Five-lemma, and observe that a theorem of Milnor (cf. \cite{Mi}) implies that the homotopy fibre of an arbitrary map between CW-complexes has the same homotopy type of a CW-complex, the second result in (3) then follows from the Whitehead theorem.
\end{proof}

Let $X$ and $Y$ be two $\mu_n$-equivariant arithmetic schemes, according to Theorem~\ref{301}. (1), to define a family of group homomorphisms
$$\widehat{K}_*(X,\mu_n)\to \widehat{K}_*(Y,\mu_n),$$
it is sufficient to give a square
$$\xymatrix{ \mid \widehat{S}(X)\mid \ar[r]^-{f} \ar[d]_-{\widetilde{{\rm ch}_g}} & \mid \widehat{S}(Y)\mid \ar[d]^-{\widetilde{{\rm ch}_g}} \\
\mid\mathcal{K}(\bigoplus_{p\geq0}\widetilde{D}^{2p-*}(X_{\mu_n},p)[-1]_{R_n})\mid \ar[r]^-{f^H} &  \mid\mathcal{K}(\bigoplus_{p\geq0}\widetilde{D}^{2p-*}(Y_{\mu_n},p)[-1]_{R_n})\mid,}$$
which commutes up to homotopy. Such a square will be denoted by $\square^X_Y$, and notice that the defined homomorphisms between the arithmetic K-groups depend on the choice of the homotopy in $\square^X_Y$.

For instance, let $i: Y\hookrightarrow X$ be a $\mu_n$-equivariant closed immersion of $\mu_n$-equivariant arithmetic schemes, denote by $j: U\to X$ the open immersion with respect to the complement $X\setminus i(Y)$. Consider the following diagram
$$\xymatrix{Fj^* \ar[r] \ar@{.>}[d]_-{\gamma_{c}} & \mid \widehat{S}(X)\mid \ar[r]^-{j^*} \ar[d]_-{\widetilde{{\rm ch}_g}} & \mid \widehat{S}(U)\mid \ar[d]^-{\widetilde{{\rm ch}_g}} \\
Fj_H^* \ar[r] & \mid \mathcal{K}(\bigoplus_{p\geq0}\widetilde{D}^{2p-*}(X_{\mu_n},p)[-1]_{R_n})\mid \ar[r]^-{j_H^*} & \mid \mathcal{K}(\bigoplus_{p\geq0}\widetilde{D}^{2p-*}(U_{\mu_n},p)[-1]_{R_n})\mid,}$$
in which the continuous maps $j^*$ and $j_H^*$ are induced by the pullbacks along $j$. Since the right square in the above diagram is commutative strictly, by Theorem~\ref{301} (1), there exists a unique continuous map $\gamma_{c}: Fj^*\to Fj_H^*$ which extends to the whole homotopy fibre sequences.

\begin{defn}\label{302}
Let notations and assumptions be as above. The equivariant higher arithmetic K-groups of $X$ with supports in $Y$ are defined as $$\widehat{K}_{Y,m}(X,\mu_n):=\pi_{m+1}\big(\text{homotopy fibre of }\gamma_{c}\big)\quad\text{for}\quad m\geq1,$$
and the equivariant regulator maps with supports,
$${\rm ch}_{Y,g}:\quad K_{Y,m}(X,\mu_n)\to \bigoplus_{p\geq0}H_{Y_{\mu_n},D}^{2p-m}(X_{\mu_n},\R(p))_{R_n}$$ are defined as the homomorphisms induced by $\gamma_{c}$ on the level of homotopy groups. Moreover, we have a family of group homomorphisms $$i_*: \widehat{K}_{Y,m}(X,\mu_n)\to \widehat{K}_m(X,\mu_n)$$
and a family of pull-back morphisms
$$j^*: \widehat{K}_m(X,\mu_n)\to \widehat{K}_m(U,\mu_n)$$
for all $m\geq1$. These pull-back morphisms are determined by the continuous map $\gamma_{j}$ between the homotopy fibres of $\widetilde{{\rm ch}_g}$.
\end{defn}

Our main result in this section is the following.

\begin{thm}\label{304}
The group homomorphism $i_*: \widehat{K}_{Y,m}(X,\mu_n)\to \widehat{K}_m(X,\mu_n)$ factors through an isomorphism $$\widehat{K}_{Y,m}(X,\mu_n)\cong \pi_{m+1}(F\gamma_{j})\quad\text{for any }m\geq 1.$$ Hence one has a long exact sequence
$$\xymatrix{\cdots\ar[r] & \widehat{K}_{Y,m}(X,\mu_n)\ar[r]^-{i_*} & \widehat{K}_m(X,\mu_n)\ar[r]^-{j^*} & \widehat{K}_m(U,\mu_n)\ar[r] & \widehat{K}_{Y,m-1}(X,\mu_n)\ar[r] & \cdots}$$
ending with
$$\xymatrix{\cdots\ar[r] & \widehat{K}_{Y,1}(X,\mu_n)\ar[r]^-{i_*} & \widehat{K}_1(X,\mu_n)\ar[r]^-{j^*} & \widehat{K}_1(U,\mu_n).}$$
\end{thm}

Theorem~\ref{304} is actually a special case of the following result.

\begin{lem}\label{305}
Consider the following square $\square$ in the category $\mathcal{V}$
$$\xymatrix{& Ff \ar@{.>}[r]^-{\alpha} \ar[d] & Fg \ar[d] \\
Fi \ar[r] \ar@{.>}[d]_-{\beta} & D \ar[r]^-{i} \ar[d]_-{f} & E \ar[d]^-{g} \\
Fj \ar[r] & A \ar[r]^-{j} & B.}$$
Suppose that $\square$ is commutative up to homotopy and fix a homotopy $H_\square$, we denote by $\alpha$ and $\beta$ the corresponding maps induced on the homotopy fibres. Then the induced group homomorphism $\pi_m(F\beta)\to \pi_m(Ff)$ factors through an isomorphism $\pi_m(F\beta)\cong \pi_{m}(F\alpha)$ for any $m\geq0$.
\end{lem}
\begin{proof}
Firstly assume that $j$ and $g$ are both fibrations. We construct the following diagram
$$\xymatrix{
  D \ar@/_/[ddr]_{f} \ar@/^/[drr]^{i}
    \ar@{.>}[dr]^-{\theta}                   \\
   & A\times_B E \ar[d]^{p_A} \ar[r]_{p_E}
                      & E \ar[d]_{g}    \\
   & A \ar[r]^{j}     & B               .}$$
Write $I$ for the interval $[0,1]$, then the homotopy $H_\square$ can be explained as a map $I\times D\to B$ such that $H_\square(0)=g\circ i$ and $H_\square(1)=j\circ f$. Since $g$ is assumed to be a fibration, we may lift $H_\square$ to get a homotopy $H: I\times D\to E$ such that $H(0)=i$ and $g\circ H(1)=H_\square(1)=j\circ f$. Then $f$ and $H(1)$ determine a map $$\theta: D\to A\times_B E$$ such that the triangle $(\theta,p_A,f)$ commutes and the triangle $(\theta,p_E,i)$ commutes up to homotopy.

Now, $p_A, p_E$ are also fibrations and we have natural homotopy equivalences $Fp_A\simeq Fg$ and $Fp_E\simeq Fj$. Moreover, applying Theorem~\ref{301}. (2) to the triangles $(\theta,p_A,f)$ and $(\theta,p_E,i)$, we get homotopy fibre sequences
$$F\theta\to Ff\to Fp_A$$ and $$F\theta \to Fi\to Fp_E.$$ Since $g\circ H=H_\square$, we have that $\beta: Fi\to Fj$ is equal to the composition $Fi\to Fp_E\simeq Fj$. This means we have the following commutative diagram
$$\xymatrix{F\theta \ar[r] \ar@{.>}[d] & Fi \ar[r] \ar[d]_-{\rm id} & Fp_E \ar[d]^-{\wr} \\
F\beta \ar[r] & Fi \ar[r]^-{\beta} & Fj,}$$
and hence a weak equivalence $\mu: F\theta\to F\beta$. On the other hand, the homotopy $H\circ\pi(f): I\times Ff\to E$ can be lifted to a homotopy $I\times Ff\to Fg$ which connects $\alpha: Ff\to Fg$ and the composition $Ff\to Fp_A\simeq Fg$, again because $g\circ H=H_\square$. Therefore, we obtain a diagram commuting up to homotopy
$$\xymatrix{F\theta \ar[r] \ar@{.>}[d] & Ff \ar[r] \ar[d]_-{\rm id} & Fp_A \ar[d]^-{\wr} \\
F\alpha \ar[r] & Ff \ar[r]^-{\alpha} & Fg,}$$ which induces a weak equivalence $\nu: F\theta\to F\alpha$. One can readily check that the two weak equivalences $\mu$ and $\nu$ provide the claimed isomorphism $\pi_m(F\beta)\cong \pi_{m}(F\alpha)$ for any $m\geq0$.

For general situation, we factor $j,g$ as a homotopy equivalence followed by a fibration and we consider the following diagram
$$\xymatrix{ D \ar[rr]^-{i} \ar[drr]^-{u} \ar[ddr]_-{v} \ar[dd]_-{f} \ar@{.>}[dr]^-{\theta} && E \ar[d]^-{\wr}_-{w} \\ & Nf\times_B N_g \ar[r] \ar[d] & N_g \ar[d]^-{s} \\ A \ar[r]^-{\sim}_-{z} & N_j \ar[r]_-{t} & B,}$$ in which $s,t$ are both fibrations and $g=s\circ w, j=t\circ z$.
Using again Theorem~\ref{301} to the commutative triangles $(i,w,u)$, $(f,z,v)$, $(w,s,g)$ and $(z,t,j)$, we get natural weak equivalences $\xymatrix{Fi \ar[r]^-{\sim_w} & Fu}$, $\xymatrix{Ff \ar[r]^-{\sim_w} & Fv}$, $\xymatrix{Fg \ar[r]^-{\sim_w} & Fs}$ and $\xymatrix{Fj \ar[r]^-{\sim_w} & Ft}$. These weak equivalences induce commutative diagrams
$$\xymatrix{F\beta \ar@{.>}[r]^-{\mu'} \ar[d]_-{\pi(\beta)} & F\beta' \ar[d]^-{\pi(\beta')} \\ Fi \ar[r]^-{\sim_w} \ar[d]_-{\beta} & Fu \ar[d]^-{\beta'} \\ Fj \ar[r]^-{\sim_w} & Ft}\quad \text{and}\quad \xymatrix{F\alpha \ar@{.>}[r]^-{\nu'} \ar[d]_-{\pi(\alpha)} & F\alpha' \ar[d]^-{\pi(\alpha')} \\ Ff \ar[r]^-{\sim_w} \ar[d]_-{\alpha} & Fv \ar[d]^-{\alpha'} \\ Fg \ar[r]^-{\sim_w} & Fs,}$$
where $\mu',\nu'$ are also weak equivalences by Theorem~\ref{301}. (3). Then the claimed isomorphism $\pi_m(F\beta)\cong \pi_{m}(F\alpha)$ follows from the weak equivalences $\mu',\nu'$ and the above arguments for fibration case, so we are done.
\end{proof}

\section{Arithmetic purity theorem}
\label{sec:4}
\subsection{The statement}
\label{sec:4.1}
As we mentioned in the introduction, for regular scheme $X$ and its regular closed subscheme $Y$ with complement $U:=X\setminus Y$, there exists a long exact sequence of algebraic K-groups $$\cdots\to K_m(Y)\to K_m(X)\to K_m(U)\to K_{m-1}(Y)\to\cdots$$
ending with $$\cdots\to K_0(Y)\to K_0(X)\to K_0(U)\to 0.$$
This long exact sequence follows from the localization theorem of Quillen. And according to the definition of algebraic K-groups with supports, it can be also deduced from a natural long exact sequence
$$\cdots\to K_{Y,m}(X)\to K_m(X)\to K_m(U)\to K_{Y,m-1}(X)\to\cdots$$
together with a family of group isomorphisms
$$K_m(Y)\cong K_{Y,m}(X)\quad\text{for }m\geq0,$$
which is called the algebraic purity theorem. In this section, we shall develop the arithmetic and the equivariant analogue of this purity theorem for $\mu_n$-equivariant arithmetic schemes. Firstly in this subsection, we formulate the statement and give the proof of Theorem~\ref{thc} as a consequence. For technical reason, we assume that $X$ is proper.

Let $i: Y\hookrightarrow X$ be a $\mu_n$-equivariant closed immersion of regular and proper $\mu_n$-equivariant arithmetic schemes over an arithmetic ring $(D,\Sigma,F_\infty)$. Denote by $U$ the complement $X\setminus i(Y)$ and by $j: U\hookrightarrow X$ the corresponding open immersion. Recall that $R_n=\R$ if $n=1$ and $R_n=\C$ otherwise.

Now, consider the commutative square $\square^X_U$
$$\xymatrix{\mid \widehat{S}(X)\mid \ar[r]^-{j^*} \ar[d]_-{\widetilde{{\rm ch}_g}} & \mid \widehat{S}(U)\mid \ar[d]^-{\widetilde{{\rm ch}_g}} \\
\mid \mathcal{K}(\bigoplus_{p\geq0}\widetilde{D}^{2p-*}(X_{\mu_n},p)[-1]_{R_n})\mid \ar[r]^-{j_H^*} & \mid \mathcal{K}(\bigoplus_{p\geq0}\widetilde{D}^{2p-*}(U_{\mu_n},p)[-1]_{R_n})\mid.}$$
We still denote by $\gamma_c: Fj^*\to Fj_H^*$ the induced map defining the arithmetic K-groups of $X$ with supports in $Y$. Moreover, we fix a $\mu_n$-invariant K\"{a}hler metric for the holomorphic tangent bundle $TX(\C)$ and endow $TY(\C)$, $TX_{\mu_n}(\C)$, $TY_{\mu_n}(\C)$ and the normal bundle $N_{X/Y}$ with the induced metrics. Write ${\rm Td}_g(\overline{N}_{X/Y})$ for the equivariant Todd form associated to the equivariant hermitian vector bundle $\overline{N}_{X/Y}$ (cf. \cite[Section 3]{KR}). By definition, for an equivariant hermitian vector bundle $\overline{E}$ on $Y$,
$${\rm Td}_g(\overline{E})={\rm Td}(\overline{E}_g)\prod_{\zeta\neq 1}{\rm det}\big(\frac{1}{1-\zeta^{-1}e^{K_\zeta}}\big)$$ where $\overline{E}_g$ is the $0$-degree part of $\overline{E}\mid_{Y_{\mu_n}}$ and $K_\zeta$ stands for the curvature form of $\overline{E}_\zeta$. Similar to the Todd form in the non-equivariant case, we have
$${\rm Td}_g(\overline{E})=\frac{c_{{\rm rk}E_g}(\overline{E}_g)}{{\rm ch}_g(\sum_{j=0}^{{\rm rk}E}(-1)^j\wedge^j\overline{E}^\vee)}.$$

Let us define a product of ${\rm Td}^{-1}_g(\overline{N}_{X/Y})$ with elements in $\widetilde{D}^*(Y_{\mu_n},p)$. For any $\alpha\in \mathfrak{D}^r(Y_{\mu_n}\times (\mathbb{P}^1)^k,p)$, $\alpha\bullet {\rm Td}^{-1}_g(\overline{N}_{X/Y}):=\alpha\bullet {\rm Td}^{-1}_g(p_0^*\overline{N}_{X/Y})$ where $p_0$ is the natural projection from $Y\times (\mathbb{P}^1)^k$ to $Y$. Because $d_i^jp_0^*=p_0^*$ and $s_ip_0^*=p_0^*$, this product induces a well-defined morphism of complexes
$$\xymatrix{\bigoplus_{p\geq0}\widetilde{D}^{2p-*}(Y_{\mu_n},p)[-1]_{R_n} \ar[rr]^-{\bullet {\rm Td}^{-1}_g(\overline{N}_{X/Y})} && \bigoplus_{p\geq0}\widetilde{D}^{2p-*}(Y_{\mu_n},p)[-1]_{R_n}}$$ and hence an isomorphism
$${\rm td}_g:\quad \mathcal{K}(\bigoplus_{p\geq0}\widetilde{D}^{2p-*}(Y_{\mu_n},p)[-1]_{R_n})\to  \mathcal{K}(\bigoplus_{p\geq0}\widetilde{D}^{2p-*}(Y_{\mu_n},p)[-1]_{R_n}).$$ We shall write $c$ for the composition
$$\xymatrix{\widehat{S}(Y) \ar[r]^-{\widetilde{{\rm ch}_g}} &  \mathcal{K}(\bigoplus_{p\geq0}\widetilde{D}^{2p-*}(Y_{\mu_n},p)[-1]_{R_n}) \ar[r]^-{{\rm td}_g} &  \mathcal{K}(\bigoplus_{p\geq0}\widetilde{D}^{2p-*}(Y_{\mu_n},p)[-1]_{R_n})}$$
and for its geometric realization. Notice that the map ${\rm td}_g$ can be defined for any (smooth at infinity) metric over $N_{X/Y}$ by making its $\mu_n$-average and the maps $c$ deduced from different choices of the metric over $N_{X/Y}$ are actually homotopic.

Next, we construct a homotopy equivalence $i_*: \mid \widehat{S}(Y)\mid\simeq Fj^*$. Let $\mathcal{M}_Y(X,\mu_n)$ denote the category of $\mu_n$-equivariant coherent sheaves on $X$ with supports in $Y$. Then, by Quillen's localization theorem,
$$\xymatrix{BQ(\mathcal{M}_Y(X,\mu_n)) \ar[r] & BQ(\mathcal{M}(X,\mu_n)) \ar[r]^-{j^*} & BQ(\mathcal{M}(U,\mu_n))}$$
is a homotopy fibre sequence. Concerning Theorem~\ref{218} and notice that $X,Y$ are regular, we have the following commutative (up to homotopy) square
$$\xymatrix{BQ(\mathcal{M}_Y(X,\mu_n)) \ar[r] \ar@{.>}[d] & BQ(\mathcal{M}(X,\mu_n)) \ar[r] \ar[d]^-{\wr} & BQ(\mathcal{M}(U,\mu_n)) \ar[d]^-{\wr} \\
Fj^* \ar[r] & \mid \widehat{S}(X)\mid \ar[r]^-{j^*} & \mid \widehat{S}(U)\mid,}$$
which induces a weak equivalence and in fact a homotopy equivalence $BQ(\mathcal{M}_Y(X,\mu_n))\simeq Fj^*$ by Theorem~\ref{301}. (3). Moreover, the theorem \textit{D\'{e}vissage} of Quillen implies that the exact functor $$i_*: \mathcal{M}(Y,\mu_n)\to \mathcal{M}_Y(X,\mu_n)$$ induces a homotopy equivalence $BQ(\mathcal{M}(Y,\mu_n))\simeq BQ(\mathcal{M}_Y(X,\mu_n))$. We then define the homotopy equivalence
$$\xymatrix{i_*:\quad\mid \widehat{S}(Y)\mid \ar[r]^-{\sim} & BQ(\mathcal{M}(Y,\mu_n))\simeq BQ(\mathcal{M}_Y(X,\mu_n)) \ar[r]^-{\sim} & Fj^*.}$$

On the other hand, to formulate the arithmetic purity theorem, we construct a homotopy equivalence $i_*^H: \mid \mathcal{K}(\bigoplus_{p\geq0}\widetilde{D}^{2p-*}(Y_{\mu_n},p)[-1]_{R_n})\mid\simeq Fj_H^*$. In the case that $Y_{\mu_n}=\emptyset$, this is trivial. For the case that $Y_{\mu_n}\neq\emptyset$, we use the fact that the real Deligne-Beilinson cohomology is a Gillet cohomology theory for smooth real varieties.

In \cite{Gi}, H. Gillet listed a series of axioms for cohomology-homology theories which include a Poincar\'{e} duality theory in the sense of Bloch-Ogus. Such axioms allow him to define Chern characters and obtain the most general form of the Riemann-Roch theorem for higher algebraic $K$-theory. Gillet's axioms contain a purity condition: if $i: Y\hookrightarrow X$ is a closed immersion of smooth real varieties of codimension $d$, then the following isomorphism (coming from the Poincar\'{e} Duality)
$$H_{\mathcal{G}}^k(Y,p)\cong H_{Y,\mathcal{G}}^{k+2d}(X,d+p)$$ is actually induced by a quasi-isomorphism $\xymatrix{i_{!}: \mathcal{G}^*(Y,p) \ar[r]^-{\sim} & Ri^{!}\mathcal{G}^*(X,d+p)[2(d+p)]}$, where $\mathcal{G}^*(\cdot,*)$ is a complex of sheaves of abelian groups providing the cohomology theory and $i^{!}$ is the functor ``sections with support in $Y$" (cf. \cite[Definition 1.2. (v),(vi),(vii)]{Gi}). Such a complex is called a Gillet complex.

As a well-known result, in \cite{Ja}, U. Jannsen has proved that the Deligne-Beilinson complex $A(p)_{D, Zar}$ which computes the Deligne-Beilinson cohomology satisfies all Gillet's axioms. In \cite{Bu1}, J. I. Burgos described a complex of vector space $D^*(X, p)$ (see Section 2.1 Corollary~\ref{203}) computing the real Deligne-Beilinson cohomology. After a sheafification by setting $\mathbf{D}^*(X, p):=D^*(X, p)$ (this presheaf is actually a sheaf), J. I. Burgos, J. Kramer and U. K\"{u}hn showed in \cite[Section 5]{BKK} that there exists a natural quasi-isomorphism $\mathbb{R}(p)_{D, Zar}\simeq \mathbf{D}^*(\cdot, p)$ so that $\mathbf{D}^*(\cdot, p)$ is also a Gillet complex. In particular, it satisfies Gillet's purity condition.

Notice that $\mathcal{K}(\bigoplus_{p\geq0}\mathbf{D}^{2p-*}(X_{\mu_n},p)[-1])$ defines a simplicial sheaf on $X_{\mu_n}$, and the usual hypercohomology of $\bigoplus_{p\geq0}\mathbf{D}^{2p-*}(X_{\mu_n},p)[-1]$ can be realized as the generalized sheaf cohomology of $\mathcal{K}(\bigoplus_{p\geq0}\mathbf{D}^{2p-*}(X_{\mu_n},p)[-1])$. In this framework, Gillet's purity condition represents an isomorphism
$$\xymatrix{i_*\mathcal{K}(\bigoplus_{p\geq0}\mathbf{D}^{2p-*}(Y_{\mu_n},p)[-1]) \ar[r]^-{\sim} & R\underline{\Gamma}_{Y_{\mu_n}}\big(\mathcal{K}(\bigoplus_{p\geq0}\mathbf{D}^{2(p+d)-*}(X_{\mu_n},p+d)[-1])\big)}$$ in the homotopy category of simplicial sheaves on $X_{\mu_n}$, where $d$ is the rank of the normal bundle $N_{X_{\mu_n}/Y_{\mu_n}}$ (cf. \cite[Section 3. (3.2)]{Gi}). 

This means, if we write ${j'}_H^*$ for the continuous map $$\mid \mathcal{K}(\bigoplus_{p\geq0}D^{2p-*}(X_{\mu_n},p)[-1]_{R_n})\mid \longrightarrow \mid \mathcal{K}(\bigoplus_{p\geq0}D^{2p-*}(U_{\mu_n},p)[-1]_{R_n})\mid,$$ there is a homotopy equivalence
$$i_*^H: \quad\mid \mathcal{K}(\bigoplus_{p\geq0}D^{2p-*}(Y_{\mu_n},p)[-1]_{R_n})\mid\simeq F{j'}_H^*.$$

Then the desired homotopy equivalence $i_*^H: \quad\mid \mathcal{K}(\bigoplus_{p\geq0}\widetilde{D}^{2p-*}(Y_{\mu_n},p)[-1]_{R_n})\mid\simeq Fj_H^*$ is the composition of a simplicial homotopy inverse of
$$\iota: \mathcal{K}(\bigoplus_{p\geq0}D^{2p-*}(Y_{\mu_n},p)[-1]_{R_n})\to \mathcal{K}(\bigoplus_{p\geq0}\widetilde{D}^{2p-*}(Y_{\mu_n},p)[-1]_{R_n}),$$
which is constructed by an explicit quasi-inverse of complexes $\varphi: \widetilde{D}^*(Y_{\mu_n},p)\to D^*(Y_{\mu_n},p)$ with the natural homotopy equivalence $F{j'}_H^*\simeq F{j}_H^*$.

We now formulate the arithmetic purity theorem whose proof will be given in next two subsections.

\begin{thm}\label{401}
Let $i: Y\hookrightarrow X$ be a $\mu_n$-equivariant closed immersion of regular and proper $\mu_n$-equivariant arithmetic schemes over an arithmetic ring $(D,\Sigma,F_\infty)$. Then the continuous maps $i_*, \gamma_c, c$ and $i_*^H$ constructed above form a square
$$\xymatrix{\mid \widehat{S}(Y)\mid \ar[r]^-{\sim}_-{i_*} \ar[d]_-{c} & Fj^* \ar[d]^-{\gamma_c} \\
\mid \mathcal{K}(\bigoplus_{p\geq0}\widetilde{D}^{2p-*}(Y_{\mu_n},p)[-1]_{R_n})\mid \ar[r]^-{\sim}_-{i_*^H} & Fj_H^*,}$$
which is commutative up to homotopy. Consequently, a fixed homotopy between $\gamma_c\circ i_*$ and $i_*^H\circ c$ gives a family of isomorphisms of abelian groups
$$i_*: \widehat{K}_m(Y,\mu_n)\cong \widehat{K}_{Y,m}(X,\mu_n)\quad\text{for any }m\geq1.$$
\end{thm}

As a corollary, we combine the arithmetic purity theorem with Theorem~\ref{304} to obtain a proof of Theorem~\ref{thc}, so that one has a long exact sequence
$$\xymatrix{\cdots\ar[r] & \widehat{K}_{m}(Y,\mu_n)\ar[r]^-{i_*} & \widehat{K}_m(X,\mu_n)\ar[r]^-{j^*} & \widehat{K}_m(U,\mu_n)\ar[r] & \widehat{K}_{m-1}(Y,\mu_n)\ar[r] & \cdots}$$
ending with
$$\xymatrix{\cdots\ar[r] & \widehat{K}_{1}(Y,\mu_n)\ar[r]^-{i_*} & \widehat{K}_1(X,\mu_n)\ar[r]^-{j^*} & \widehat{K}_1(U,\mu_n).}$$

\subsection{The case of zero section embedding}
\label{sec:4.2}
As before, let $Y$ be a regular $\mu_n$-equivariant arithmetic scheme which is proper over an arithmetic ring $(D,\Sigma,F_\infty)$. Let $N$ be a $\mu_n$-equivariant vector bundle on $Y$, we denote by $P$ the projective completion of $E$ i.e. the projective space bundle $\mathbb{P}(N\oplus \mathcal{O}_Y)$. Write $\pi: \mathbb{P}(N\oplus \mathcal{O}_Y)\to Y$ for the canonical smooth projection, then the morphism
$\mathcal{O}_Y\to N\oplus \mathcal{O}_Y$ induces a canonical section $i_\infty: Y\hookrightarrow P$ which is called the zero section embedding. In this subsection, we shall prove Theorem~\ref{401} in the special situation where $X=P$ and $i=i_\infty$. The reason for us to focus on this zero section embedding is that the direct image $i_*E$ of every vector bundle $E$ on $Y$ admits a canonical resolution on $P$ which is called the Koszul resolution, so that to construct an explicit simplicial map from the K-theory space of $Y$ to the K-theory space of $P$ becomes possible.

In fact, on $P=\mathbb{P}(N\oplus \mathcal{O}_Y)$, there exists a tautological exact sequence
$$0\to \mathcal{O}(-1)\to \pi^*(N\oplus \mathcal{O}_Y)\to Q\to 0$$
where $Q$ is the tautological quotient bundle. This exact sequence and the inclusion $\mathcal{O}_P\to \pi^*(N\oplus \mathcal{O}_Y)$ induces a section $\sigma: \mathcal{O}_P\to Q$ which vanishes along the zero section $i(Y)$. By duality we get a morphism $Q^\vee\to \mathcal{O}_P$ and this morphism induces an exact sequence
$$0\to \wedge^nQ^\vee\to \cdots \to \wedge^2Q^\vee\to Q^\vee\to \mathcal{O}_P\to i_*\mathcal{O}_Y\to 0$$ where $n={\rm rank}(Q)$. Generally, by projection formula, any vector bundle $E$ on $Y$ admits a canonical resolution
$$0\to \wedge^nQ^\vee\otimes\pi^*E\to \cdots \to \wedge^2Q^\vee\otimes\pi^*E\to Q^\vee\otimes\pi^*E\to \pi^*E\to i_*E\to 0$$
on $P$. This is called the Koszul resolution and it will be denoted by $K(E,N)$.

Write again $j: U\hookrightarrow P$ for the open immersion with $U:=P\setminus i(Y)$, we want to prove that the following square constructed in last subsection
$$\xymatrix{\mid \widehat{S}(Y)\mid \ar[r]^-{\sim}_-{i_*} \ar[d]_-{c} & Fj^* \ar[d]^-{\gamma_c} \\
\mid \mathcal{K}(\bigoplus_{p\geq0}D^{2p-*}(Y_{\mu_n},p)[-1]_{R_n})\mid \ar[r]^-{\sim}_-{i_*^H} & Fj_H^*}$$
is commutative up to homotopy. It is equivalent to say that $c=(i_*^H)^{-1}\circ \gamma_c\circ i_*$ in the homotopy category $Ho(\mathcal{V})$. Here we use $D^{2p-*}(Y_{\mu_n},p)$ instead of $\widetilde{D}^{2p-*}(Y_{\mu_n},p)$ and therefore $c$ is deduced from the morphism of complexes ${\rm Td}^{-1}_g(\overline{N})\bullet {\rm ch}_g(\cdot)$.

Since $Y$ (and hence $P$) is supposed to be proper over $D$, we shall replace Deligne-Beilinson cohomology groups by Deligne homology groups according to the Poincar\'{e} duality. Firstly, let $d$ be the dimension of $Y(\C)$ and write $B_{Y(\C)}^*$ for the complex of real vector spaces ${'E}_{Y(\C)}^*[-2d](-d)$. Replacing the complex $E^*_{\log}$ in Section 2.1 Theorem~\ref{202} by $B^*$, we may construct a homological complex ${'D}^{2p-*}(Y,p)$ (similar notation to $D^{2p-*}(\cdot,p)$ used in Section 2.1 Corollary~\ref{203}). This complex computes the real Deligne homology groups of $Y$. Moreover, the morphism of complexes $[\cdot]: E_{Y(\C)}^*\to B_{Y(\C)}^*$ is a filtered quasi-isomorphism compatible with underlying real structures, then one gets a quasi-isomorphism
$$\xymatrix{D^{2p-*}(Y,p) \ar[r]^-{\sim} & {'D}^{2p-*}(Y,p)}$$ and hence a homotopy equivalence
$$\xymatrix{\mid \mathcal{K}(\bigoplus_{p\geq0}D^{2p-*}(Y,p)[-1])\mid \ar[r]^-{\sim} & \mid \mathcal{K}(\bigoplus_{p\geq0} {'D}^{2p-*}(Y,p)[-1])\mid.}$$ The same constructions and results go to $Y_{\mu_n}, P$ and $P_{\mu_n}$.

Notice that to any proper morphism $f: U\to V$ between complex algebraic manifolds of relative dimension $e$ we may associate a pushforward
$$f_{!}: {'E}_U^*\to {'E}_V^*[-2e](-e)$$ by setting $(f_{!}T)(\eta)=T(f^*\eta)$. Then, according to the construction of the Poincar\'{e} duality, the composition $$\mid \mathcal{K}(\bigoplus_{p\geq0}D^{2p-*}(Y_{\mu_n},p)[-1]_{R_n})\mid\simeq Fj_H^*\to \mid \mathcal{K}(\bigoplus_{p\geq0}D^{2p-*}(P_{\mu_n},p)[-1]_{R_n})\mid$$
which is still denoted by $i_*^H$, can be viewed as the geometric realization of a simplicial map coming from the morphism of complexes
$${i_{\mu_n}}_!: \bigoplus_{p\geq0}{'D}^{2p-*}(Y_{\mu_n},p)[-1]\to \bigoplus_{p\geq0}{'D}^{2p-*}(P_{\mu_n},p)[-1].$$

On the other hand, the morphism of complexes
$${\pi_{\mu_n}}_!: \bigoplus_{p\geq0}{'D}^{2p-*}(P_{\mu_n},p)[-1]\to \bigoplus_{p\geq0}{'D}^{2p-*}(Y_{\mu_n},p)[-1]$$
induces a continuous map
$$\pi_*^H: \quad \mid \mathcal{K}(\bigoplus_{p\geq0}{'D}^{2p-*}(P_{\mu_n},p)[-1]_{R_n})\mid\to \mid \mathcal{K}(\bigoplus_{p\geq0}{'D}^{2p-*}(Y_{\mu_n},p)[-1]_{R_n})\mid$$
such that $\pi_*^H\circ i_*^H={\rm Id}$ because $i_{\mu_n}$ is a section.

Therefore, to prove that $c=(i_*^H)^{-1}\circ \gamma_c\circ i_*$ in the homotopy category $Ho(\mathcal{V})$ is equivalent to show that the following square
$$\xymatrix{\mid \widehat{S}(Y)\mid \ar[r]^-{\sim}_-{i_*} \ar[d]_-{c} & Fj^* \ar[r] & \mid \widehat{S}(P)\mid \ar[d]^-{\widetilde{{\rm ch}_g}}\\
\mid \mathcal{K}(\bigoplus_{p\geq0}{'D}^{2p-*}(Y_{\mu_n},p)[-1]_{R_n})\mid \ar[rr]^-{i_*^H} && \mid \mathcal{K}(\bigoplus_{p\geq0}{'D}^{2p-*}(P_{\mu_n},p)[-1]_{R_n})\mid}$$
is commutative up to homotopy.

Furthermore, we simplify our discussion by using the K-theory of derived categories. Firstly, we denote by $\widehat{\mathcal{P}E}(P,\text{quasi-iso})$ the Waldhausen category of bounded (homological) complex of hermitian $\mu_n$-equivariant vector bundles on $P$, with that the weak equivalences are quasi-isomorphisms. And similarly, we denote by $\widehat{\mathcal{P}E}(P,\text{iso})$ the Waldhausen category of bounded (homological) complex of hermitian $\mu_n$-equivariant vector bundles on $P$, with that the weak equivalences are isomorphisms. Cofibrations in both Waldhausen categories are degree-wise admissible monomorphisms (cf. \cite[1.11.6]{TT}).Then there is an exact functor $\widehat{\mathcal{P}}(P)\to \widehat{\mathcal{P}E}(P,\text{quasi-iso})$, which sends a hermitian $\mu_n$-equivariant vector bundle $\overline{F}$ on $P$ to the complex which is $\overline{F}$ in degree $0$ and $0$ in other degrees. It is clear that this exact functor factors as $$\widehat{\mathcal{P}}(P)\to \widehat{\mathcal{P}E}(P,\text{iso})\to\widehat{\mathcal{P}E}(P,\text{quasi-iso}).$$

As before, we construct a simplicial set $[n]\mapsto {\rm Ob}\big({\rm w}S_n\widehat{\mathcal{P}E}(P,\text{iso})\big)$ and denote it by $\widehat{SE}(P)$. 

The following theorem summaries necessary properties of Waldhausen K-theory spaces that we need for later discussion.

\begin{thm}\label{403}
The canonical exact functor $\widehat{\mathcal{P}}(P)\to \widehat{\mathcal{P}E}(P,\text{quasi-iso})$ induces a homotopy equivalence of Waldhausen K-theory spaces
$$\xymatrix{\varrho:\quad\mid\widehat{S}(P)\mid \ar[r]^-{\sim} & \mid N.{\rm w}S.\widehat{\mathcal{P}E}(P,\text{quasi-iso})\mid,}$$
and one may choose a homotopy inverse $\varrho^{-1}$ of $\varrho$ such that the composition
$$\xymatrix{ \mid\widehat{SE}(P)\mid \ar[r]^-{e} & \mid N.{\rm w}S.\widehat{\mathcal{P}E}(P,\text{quasi-iso})\mid \ar[r]^-{\varrho^{-1}} & \mid\widehat{S}(P)\mid}$$ is homotopic to the Euler characteristic map, which sends $\overline{F}.$ to $\sum(-1)^k\overline{F}_k$.
\end{thm}
\begin{proof}
The first statement is well known and it follows from Theorem~\ref{218} and \cite[Theorem 1.11.7]{TT}, and also from \cite[Theorem 6.2, Lemma 6.3]{Gi}. The second statement is a byproduct of the proof of \cite[Theorem 1.11.7]{TT}, see \cite[(1.11.7.7) and (1.11.7.9)]{TT}. If we write $\widehat{\mathcal{P}E}_a^b(P,\text{iso})$ for the full subcategory of those complexes $\overline{F}.$ such that $\overline{F}_k=0$ for $k\leq a-1$ and for $k\geq b+1$, then the Euler characteristic map should be understood as a homotopy equivalence (deduced from the additivity theorem) followed by the usual Euler characteristic map, namely
\begin{align*}
   \mid\widehat{SE}_a^b(P)\mid \simeq \prod^{b-a+1}\mid\widehat{S}(P)\mid & \longrightarrow \mid\widehat{S}(P)\mid \\
   (x_b,\ldots,x_a) & \longmapsto \sum(-1)^kx_k.
\end{align*}
For this map, see also \cite[p. 259]{Gi}.
\end{proof}

Now, we may write down a diagram
$$\xymatrix{\mid \widehat{S}(P)\mid \ar[r] \ar[dd]_-{\widetilde{{\rm ch}_g}} & \mid\widehat{SE}(P)\mid \ar[r] \ar[ldd]_-{\Sigma\widetilde{{\rm ch}_g}} & \mid N.{\rm w}S.\widehat{\mathcal{P}E}(P,\text{quasi-iso})\mid \ar[lldd]^-{\widetilde{{\rm ch}_g}\circ \varrho^{-1}}\\
\\
\mid \mathcal{K}(\bigoplus_{p\geq0}D^{2p-*}(P_{\mu_n},p)[-1]_{R_n})\mid,}$$
in which the map $\Sigma\widetilde{{\rm ch}_g}$ is defined in a similar way to $\widetilde{{\rm ch}_g}$. Precisely, suppose that we are given a cube of complexes of hermitian bundles, we will regard it as a complex of hermitian cubes and associate to it the alternating sum of the equivariant higher Bott-Chern forms of its elements. Then, according to Theorem~\ref{403}, we know that the left triangle in the above diagram commutes, the right triangle in the above diagram commutes up to homotopy, and the whole triangle $(\varrho,\widetilde{{\rm ch}_g}\circ \varrho^{-1},\widetilde{{\rm ch}_g})$ commutes up to homotopy.

Let us define a continuous map $\mathfrak{i}_*: \mid \widehat{S}(Y)\mid \to \mid\widehat{SE}(P)\mid$ as the geometric realization of a simplicial map coming from the exact functor $$\mathfrak{i}_*: \widehat{\mathcal{P}}(Y)\to \widehat{\mathcal{P}E}(P)$$ which sends a hermitian $\mu_n$-equivariant vector bundle $\overline{F}$ on $Y$ to its Koszul resolution $K(\overline{F},\overline{N})$ and the metric on the tautological quotient bundle $Q$ is supposed to satisfy the Bismut's assumption (A) (cf. \cite[Section 3.4]{KR} or see below). Then the map $i_*: \mid \widehat{S}(Y)\mid \to \mid\widehat{S}(P)\mid$ can be viewed, up to homotopy, as the composition $\varrho^{-1}\circ e\circ \mathfrak{i}_*$. On the other hand, we may associate a simplicial map $\mathfrak{i}_*: \mathbb{Z}\widehat{S}(Y)\to \mathbb{Z}\widehat{S}(P)$ to $i: Y\to P$ by sending a simplex $\sigma$ to $\sum(-1)^j\overline{Q}^\vee\otimes \sigma$. Then we have a natural commutative diagram

$$\xymatrix{\mid \widehat{S}(Y)\mid \ar[r]^-{i_*} \ar[d]_-{{\rm Hu}} & \mid\widehat{S}(P)\mid \ar[d]^-{{\rm Hu}}\\
\mid \mathbb{Z}\widehat{S}(Y)\mid \ar[r]^-{\mathfrak{i}_*} & \mid \mathbb{Z}\widehat{S}(P)\mid}$$
up to homotopy. Moreover, we have a commutative diagram of simplicial abelian groups
$$\xymatrix{\mathbb{Z}\widehat{S}(Y) \ar[r]^-{\mathfrak{i}_*} \ar[d]_-{\vartheta} & \mathbb{Z}\widehat{S}(P) \ar[d]^-{\vartheta}\\
\mathbb{Z}\widehat{S}(Y,\mu_n) \ar[r]^-{\mathfrak{i}_*} & \mathbb{Z}\widehat{S}(P,\mu_n).}$$

Therefore, to prove that $c$ is equal to $(i_*^H)^{-1}\circ \gamma_c\circ i_*$ in the homotopy category $Ho(\mathcal{V})$ is equivalent to show that the following square
\begin{align}\label{a2}
\xymatrix{\mathbb{Z}\widehat{S}(Y,\mu_n) \ar[r]^-{\mathfrak{i}_*} \ar[d]_-{{\rm Td}_g^{-1}(\overline{N})\bullet \widetilde{{\rm ch}_g}} & \mathbb{Z}\widehat{S}(P,\mu_n) \ar[d]^-{\widetilde{{\rm ch}_g}}\\
\mathcal{K}(\bigoplus_{p\geq0}{'D}^{2p-*}(Y_{\mu_n},p)[-1]_{R_n}) \ar[r]^-{i_*^H} & \mathcal{K}(\bigoplus_{p\geq0}{'D}^{2p-*}(P_{\mu_n},p)[-1]_{R_n})}
\end{align}
is commutative up to simplicial homotopy.

To construct an explicit simplicial homotopy for the square (\ref{a2}), we recall the equivariant Bott-Chern singular current due to Bismut.
This construction was realized via some current valued zeta function which involves the supertraces of Quillen's superconnections.

Let $i: Y\rightarrow X$ be a $\mu_n$-equivariant closed immersion of compact quasi-projective K\"{a}hler manifolds with hermitian normal bundle $\overline{N}$. We denote by $g$ a generator of $\mu_n(\C)$ and fix a primitive $n$-th root of unity $\zeta_n$ corresponding the action of $g$. Let $\overline{\eta}$ be a hermitian $g$-bundle on $Y$ and let $\overline{\xi}.$ be a complex of hermitian $g$-bundles on $X$ such that the underlying complex of $g$-bundles $\xi.$ provides a projective resolution of $i_*\eta$. We denote the differential of the complex $\xi.$ by $v$. Notice that $\xi.$ is acyclic outside $Y$ and the homology sheaves of its restriction to $Y$ are locally free. We write $H_n=\mathcal{H}_n(\xi.\mid_Y)$ and
define a $\Z$-graded bundle $H=\bigoplus_nH_n$. For $y\in Y$ and $u\in TX_y$, we denote by $\partial_uv(y)$ the derivative of $v$ at $y$ in the direction $u$ in any given holomorphic trivialization of $\xi.$ near $y$. Then the map $\partial_uv(y)$ acts on $H_y$ as a chain map, and this action only depends on the image $z$ of $u$ in $N_y$. So we get a chain complex of holomorphic vector bundles $(H,\partial_zv)$.

Let $\pi$ be the projection from the normal bundle $N$ to $Y$, we have a canonical identification of $\Z$-graded chain complexes
\begin{displaymath}
(\pi^*H,\partial_zv)\cong(\pi^*(\wedge^\bullet
N^\vee\otimes\eta),\sqrt{-1}i_z).
\end{displaymath}
Moreover, such an identification is an identification of $g$-bundles. By finite dimensional Hodge theory, for each $y\in Y$, there is a canonical
isomorphism
\begin{displaymath}
H_y\cong\{f\in \xi._y\mid vf=0, v^*f=0\}
\end{displaymath}
where $v^*$ is the dual of $v$ with respect to the metrics on $\xi.$. This means that $H$ can be regarded as a smooth $\Z$-graded $g$-equivariant subbundle of $\xi$ so that it carries an induced $g$-invariant metric. On the other hand, we endow $\wedge^\bullet N^\vee\otimes \eta$ with the metric induced by $\overline{N}$ and by $\overline{\eta}$.

\begin{defn}\label{BC810}
We say that the metrics on the complex of hermitian $g$-vector bundles $\overline{\xi}.$ satisfy Bismut's assumption (A) if
the identification $(\pi^*H,\partial_zv)\cong(\pi^*(\wedge^\bullet
N^\vee\otimes\eta),\sqrt{-1}i_z)$ also identifies the metrics.
\end{defn}

\begin{prop}\label{BC811}
There always exist $g$-invariant metrics on $\xi.$ which satisfy
Bismut's assumption (A) with respect to $\overline{N}$ and
$\overline{\eta}$.
\end{prop}
\begin{proof}
This is \cite[Proposition 3.5]{Bi}.
\end{proof}

Let $\nabla^{\xi}$ be the canonical hermitian holomorphic connection on $\xi.$, then for $u>0$, we may define a $g$-invariant superconnection \begin{displaymath}
A_u:=\nabla^\xi+\sqrt{u}(v+v^*)
\end{displaymath}
on the $\Z_2$-graded vector bundle $\xi$. On the other hand, we define a superconnection on $H$
\begin{displaymath}
B:=\nabla^H+\partial_z v+(\partial_z v)^*.
\end{displaymath}

\begin{lem}\label{BC812}
Let $N_H$ be the number operator on the complex $\xi.$ i.e. it
acts on $\xi_j$ as multiplication by $j$, then for $s\in \C$ and
$0< {\rm Re}(s)<\frac{1}{2}$, the current valued zeta function
\begin{displaymath}
Z_g(\overline{\xi}.)(s):=\frac{1}{\Gamma(s)}\int_0^\infty
u^{s-1}\{[{\rm Tr_s}(N_Hg{\rm exp}(-A_u^2))]-[\int_{N_g}{\rm Tr_s}(N_Hg{\rm exp}(-B^2))]\delta_{Y_g}\}{\rm d}u
\end{displaymath}
is well-defined on $X_g$ and it has a meromorphic continuation to the complex plane which is holomorphic at $s=0$.
\end{lem}

\begin{defn}\label{BC813}
The equivariant Bott-Chern singular current on $X_g$ associated to
the resolution $\overline{\xi}.$ is defined as
\begin{displaymath}
T_g(\overline{\xi}.):=-\frac{1}{2}\frac{\partial}{\partial
s}\mid_{s=0}Z_g(\overline{\xi}.)(s).
\end{displaymath}
\end{defn}

\begin{thm}\label{BC814}
The current $T_g(\overline{\xi}.)$ belongs to $\oplus_{p\geq0}{'D}^{2p-1}(X_g,p)$ and it satisfies the differential equation
\begin{displaymath}
d(T_g(\overline{\xi}.))=\sum_k(-1)^k{\rm ch}_g(\overline{\xi}_k)-{i_{\mu_n}}_!\big({\rm
ch}_g(\overline{\eta}){\rm
Td}_g^{-1}(\overline{N})\big).
\end{displaymath}
\end{thm}

Now, let us go back to the square (\ref{a2}). Denote by $i_Y$ (resp. $i_P$) the closed immersion $Y_{\mu_n}\hookrightarrow Y$ (resp. $P_{\mu_n}\hookrightarrow P$) with respect to the fixed point subscheme $Y_{\mu_n}$ (resp. $P_{\mu_n}$). Constructing the embedding morphism $\mathfrak{i}_*$ for hermitian $k$-cubes in the same way as for $(k+1)$-simplices, after that we consider the following diagram
$$\xymatrix{ \Z\widehat{S}(Y,\mu_n) \ar[r]^-{{\rm Cub}} \ar[d]^-{\mathfrak{i}_*} & \mathcal{K}(\widehat{\Z}C_*(Y,\mu_n)[-1]) \ar[r]^-{i_Y^*} \ar[d]^-{\mathfrak{i}_*} & \mathcal{K}(\widehat{\Z}C_*(Y_{\mu_n},\mu_n)[-1]) \\
\Z\widehat{S}(P,\mu_n) \ar[r]^-{{\rm Cub}} & \mathcal{K}(\widehat{\Z}C_*(P,\mu_n)[-1]) \ar[r]^-{i_P^*}
& \mathcal{K}(\widehat{\Z}C_*(P_{\mu_n},\mu_n)[-1]).}$$
It is clear that the left square in the above diagram is naturally commutative, so to prove Theorem~\ref{401} for the case of zero section embedding, we are left to prove the following proposition.

\begin{prop}\label{404}
The diagram
$$\xymatrix{\widehat{\Z}C_*(Y,\mu_n) \ar[r]^-{i_Y^*} \ar[d]^-{\mathfrak{i}_*} & \widehat{\Z}C_*(Y_{\mu_n},\mu_n) \ar[r]^-{{\rm ch}_g} & \bigoplus_{p\geq0}{'D}^{2p-*}(Y_{\mu_n},p)_{R_n} \ar[d]^-{{i_{\mu_n}}_!\circ \big((\cdot)\bullet {\rm Td}^{-1}_g(\overline{N}_{P/Y})\big)} \\
\widehat{\Z}C_*(P,\mu_n) \ar[r]^-{i_P^*} & \widehat{\Z}C_*(P_{\mu_n},\mu_n) \ar[r]^-{{\rm ch}_g} & \bigoplus_{p\geq0}{'D}^{2p-*}(P_{\mu_n},p)_{R_n}}$$
is commutative up to homotopy of chain complexes. A homotopy between ${\rm ch}_g\circ i_P^* \circ \mathfrak{i}_*$ and ${i_{\mu_n}}_!\circ \big({\rm ch}_g\bullet {\rm Td}_g(\overline{N}_{P/Y})\big)\circ i_Y^*$ is given by
$$\mathbf{H}_k(\overline{F})=T_g\big(K(\overline{O}_Y,\overline{N})\big)\bullet \pi_{\mu_n}^*\big({\rm ch}_g(i_Y^*\overline{\mathcal{F}})\big)$$
where $T_g\big(K(\overline{O}_Y,\overline{N})\big)$ is the Bott-Chern singular current associated to the Koszul resolution $K(\overline{O}_Y,\overline{N})$.
\end{prop}
\begin{proof}
Write ${\rm ch}_g={\rm ch}_g\circ i_Y^*$ and ${\rm ch}_g={\rm ch}_g\circ i_P^*$ for short. Observe that the morphism $\lambda$ commutes with the pull-back maps and with the pushforward $\mathfrak{i}_*$. We compute, for a hermitian $k$-cube $\overline{\mathcal{F}}$,
\begin{align*}
&{\rm ch}_g(\mathfrak{i}_*\overline{\mathcal{F}})-{i_{\mu_n}}_!\big({\rm ch}_g(\overline{\mathcal{F}})\bullet {\rm Td}^{-1}_g(\overline{N}_{P/Y})\big)\\
=& {\rm ch}_g(\pi^*\overline{\mathcal{F}})\bullet\sum_{l=0}^n(-1)^l{\rm ch}_g(\wedge^l\overline{Q}^\vee)-{i_{\mu_n}}_!\big(i_{\mu_n}^*\pi_{\mu_n}^*{\rm ch}_g(\overline{\mathcal{F}})\bullet {\rm Td}^{-1}_g(\overline{N}_{P/Y})\big)\\
=& [\sum_{l=0}^n(-1)^l{\rm ch}_g(\wedge^l\overline{Q}^\vee)-{i_{\mu_n}}_!\big({\rm Td}^{-1}_g(\overline{N}_{P/Y})\big)]\bullet \pi_{\mu_n}^*{\rm ch}_g(\overline{\mathcal{F}})\\
=& d\big(T_g(K(\overline{O}_Y,\overline{N}))\big)\bullet \pi_{\mu_n}^*{\rm ch}_g(\overline{\mathcal{F}})-\mathbf{H}_k(d\overline{\mathcal{F}})+\mathbf{H}_k(d\overline{\mathcal{F}})\\
=& d\big(T_g(K(\overline{O}_Y,\overline{N}))\big)\bullet \pi_{\mu_n}^*{\rm ch}_g(\overline{\mathcal{F}})-T_g(K(\overline{O}_Y,\overline{N}))\bullet d\big(\pi_{\mu_n}^*{\rm ch}_g(\overline{\mathcal{F}})\big)+\mathbf{H}_k(d\overline{\mathcal{F}})\\
=&d\big(\mathbf{H}_k(\overline{F})\big)+\mathbf{H}_k(d\overline{\mathcal{F}}).
\end{align*}
So we are done.
\end{proof}

To end this subsection, let us introduce a push-forward map
$$({i_{\mu_n}}_!)': \bigoplus_{p\geq 0}D^{2p-*}(Y_{\mu_n},p)[-1]\to \bigoplus_{p\geq 0}D^{2p-*}(P_{\mu_n},p)[-1]$$
in order to produce an arithmetic self intersection formula in the last section. Firstly, notice that the fixed point submanifold of
$\mathbb{P}(N\oplus\mathcal{O}_Y)$ is $\mathbb{P}(N_{g}\oplus \mathcal{O}_{Y_{\mu_n}})\coprod_{\zeta\neq1} \mathbb{P}(N_\zeta)$ and the closed immersion $i_{\mu_n}$ factors through $\mathbb{P}(N_{g}\oplus \mathcal{O}_{Y_{\mu_n}})$  whose image doesn't meet the other
components $\mathbb{P}(N_\zeta)$. Moreover the
complex $K(\mathcal{O}_{Y},N)_g$, obtained by taking the
$0$-degree part of the Koszul resolution $K(\mathcal{O}_{Y},N)$ provides a resolution of $\mathcal{O}_{Y_{\mu_n}}$ on $P_{\mu_n}$ (cf. \cite[Section 2 and 6.2]{KR}). We denote $P_0:=\mathbb{P}(N_{g}\oplus\mathcal{O}_{Y_{\mu_n}})$ and $r:={\rm rank}(Q_g\mid_{P_0})$, then $c_r(\overline{Q}_g\mid_{P_0})$ is a well-defined smooth differential form in $D^{2r}(P_{\mu_n},r)$. We define
\begin{align*}
({i_{\mu_n}}_!)': \bigoplus_{p\geq 0}D^{2p-*}(Y_{\mu_n},p)[-1] & \longrightarrow \bigoplus_{p\geq 0}D^{2p-*}(P_{\mu_n},p)[-1] \\
x & \longmapsto \pi_{\mu_n}^*(x)\bullet c_r(\overline{Q}_g\mid_{P_0}).
\end{align*}
This is a morphism of complexes since the Chern form $c_r$ is closed. Now, we obtain a square
\begin{align}\label{a3}
\xymatrix{\bigoplus_{p\geq0}{D}^{2p-*}(Y_{\mu_n},p)[-1] \ar[r]^-{({i_{\mu_n}}_!)'} \ar[d]_-{[\cdot]} & \bigoplus_{p\geq0}{D}^{2p-*}(P_{\mu_n},p)[-1] \ar[d]^-{[\cdot]}\\
\bigoplus_{p\geq0}{'D}^{2p-*}(Y_{\mu_n},p)[-1] \ar[r]^-{{i_{\mu_n}}_!} & \bigoplus_{p\geq0}{'D}^{2p-*}(P_{\mu_n},p)[-1]}
\end{align}
which is commutative up to chain homotopy. An explicit homotopy is given by the Euler-Green current $\widetilde{e}(P_0,\overline{Q}_g\mid_{P_0},i_{\mu_n})=T_g\big(K(\overline{O}_Y,\overline{N})\big)\bullet{\rm Td}_g(\overline{Q})\in \bigoplus_{r\geq0}{'D}^{2r-1}(P_{\mu_n},r)$  associated to the non-equivariant zero section embedding $i_{\mu_n}: Y_{\mu_n}\hookrightarrow P_0$. This Euler-Green current satisfies the differential equation
$$d\big(\widetilde{e}(P_0,\overline{Q}_g\mid_{P_0},i_{\mu_n})\big)=[c_r(\overline{Q}_g\mid_{P_0})]-\delta_{Y_{\mu_n}}.$$
As a consequence of the commutativity (up to homotopy) of the square (\ref{a3}), the continuous map
$$(i_*^H)':\quad \mid \mathcal{K}(\bigoplus_{p\geq0}{D}^{2p-*}(Y_{\mu_n},p)[-1]_{R_n})\mid \to \mid \mathcal{K}(\bigoplus_{p\geq0}{D}^{2p-*}(P_{\mu_n},p)[-1]_{R_n})\mid$$
induces by $({i_{\mu_n}}_!)': \bigoplus_{p\geq0}{D}^{2p-*}(Y_{\mu_n},p)[-1] \to \bigoplus_{p\geq0}{D}^{2p-*}(P_{\mu_n},p)[-1]$ is homotopic to
$${i_*^H}: \quad \mid \mathcal{K}(\bigoplus_{p\geq0}{D}^{2p-*}(Y_{\mu_n},p)[-1]_{R_n})\mid \to \mid \mathcal{K}(\bigoplus_{p\geq0}{D}^{2p-*}(P_{\mu_n},p)[-1]_{R_n})\mid$$ induced by
${i_{\mu_n}}_!: \bigoplus_{p\geq0}{'D}^{2p-*}(Y_{\mu_n},p)[-1] \to \bigoplus_{p\geq0}{'D}^{2p-*}(P_{\mu_n},p)[-1]$.

Finally, notice that the metric of $\overline{Q}^\vee$ is supposed to satisfy Bismut's assumption (A), then $\overline{Q}^\vee\mid_Y$ is equivariantly isometric to $\overline{N}^\vee$ and we have
\begin{align*}
i_{\mu_n}^*({i_{\mu_n}}_!)'\big({\rm
Td}_g^{-1}(\overline{N})\big)&=i_{\mu_n}^*\big(\pi_{\mu_n}^*({\rm
Td}_g^{-1}(\overline{N}))\bullet c_r(\overline{Q}_g\mid_{P_0})\big)\\
&={\rm
Td}_g^{-1}(\overline{N})\bullet c_r(\overline{N}_g)\\
&={\rm ch}_g(\sum_{j=0}^{{\rm rk}N}(-1)^j\wedge^j\overline{N}^\vee).
\end{align*}

\subsection{Proof of the statement}
\label{se:4.3}
In this subsection, we finish the proof of Theorem~\ref{401}, so let assumptions and notations be as in Section~\ref{sec:4.1}. We shall use the deformation to the normal cone construction, this is a general technique to deal with Riemann-Roch problems. Denote $S={\rm Spec}(D)$ and let $W$ be the blow up of $X\times \mathbb{A}_S^1$ along $Y\times\{0\}$, we have a diagram
$$\xymatrix{ Y\times\mathbb{A}_S^1 \ar[r]^-{l} \ar[d]^-p \ar[rd] & W \ar[d]^-{\pi} \\
Y & \mathbb{A}_S^1}$$
such that $l$ is a closed immersion and for $t\neq0$ the inclusion $l_t: Y\times \{t\} \hookrightarrow W_t=\pi^{-1}(t)\cong X\times\{t\}$ is induced by the inclusion $i: Y\hookrightarrow X$. For $t=0$, $W_0=\pi^{-1}(0)$ is isomorphic to $\mathbb{P}(N_{X/Y}\oplus \mathcal{O}_Y)\cup \widetilde{X}$ with $\widetilde{X}$ the blow up of $X$ along $Y$ and the image $l_0(Y\times\{0\})$ (in fact $l(Y\times \mathbb{A}_S^1)$) doesn't meet $\widetilde{X}$. A core property of the deformation to the normal cone construction is that the two squares in the following deformation diagram are both Tor-independent.
$$\xymatrix{ Y\times\{t\} \ar[r]^-{s_t} \ar[d]^-{l_t} & Y\times \mathbb{A}_S^1 \ar[d]^-{l} & Y\times\{0\} \ar[l]_-{s_0}  \ar[d]^-{l_0} \\
X\times\{t\} \ar[r] & W & \mathbb{P}(N_{X/Y}\oplus \mathcal{O}_Y). \ar[l]}$$

Regard $\mathbb{A}_S^1$ as a $\mu_n$-equivariant scheme with the trivial action and notice that the deformation to the normal cone construction also works in the $\mu_n$-equivariant case, we apply all constructions before the statement of Theorem~\ref{401} to the closed immersion $l: Y\times\mathbb{A}_S^1\to W$. To this, we fix a smooth at infinity K\"{a}hler metric on $W(\C)$ and endow all its submanifolds with the induced metrics. Notice that we have $N_{W/{\mathbb{A}_Y^1}}\mid_{W_t}\cong N_{W_t/Y}$.

\begin{lem}\label{405}
Let $Z=X$ if $t\neq0$ and $Z=\mathbb{P}(N_{X/Y}\oplus \mathcal{O}_Y)$ if $t=0$. Let the symbol $\mathcal{K}(D(\cdot))$ stand for the simplicial abelian group $\mathcal{K}(\bigoplus_{p\geq0}\widetilde{D}^{2p-*}((\cdot)_{\mu_n},p)[-1]_{R_n})$. Then the continuous maps ${\gamma_c}_t\circ s_t^*$ and $s_t^*\circ \gamma_c$ deduced in the following diagram
$$\xymatrix@!0{
  && Fj^* \ar[rrrr] \ar@{.>}[ddll]_-{\gamma_c} \ar@{.>}'[dd][dddd]^-{s_t^*}
      &&  && \mid \widehat{S}(W)\mid \ar[rrrr] \ar[ddll] \ar'[dd][dddd]
      &&  && \mid \widehat{S}(W\setminus{\mathbb{A}_Y^1})\mid \ar[ddll] \ar[dddd]        \\ \\
  F{j}_H^* \ar[rrrr] \ar@{.>}[dddd]^-{s_t^*}
      &&  && \mid \mathcal{K}(DW)\mid \ar[rrrr] \ar[dddd]
      &&  && \mid \mathcal{K}(D(W\setminus{\mathbb{A}_Y^1}))\mid \ar[dddd] \\ \\
  && F{j}_t^* \ar@{.>}[ddll]^-{{\gamma_c}_t} \ar'[rr][rrrr]
      &&  && \mid \widehat{S}(Z)\mid \ar[ddll] \ar'[rr][rrrr]
      &&  && \mid \widehat{S}(Z\setminus Y)\mid  \ar[ddll]            \\ \\
  F{j_t}_H^* \ar[rrrr]
      &&  && \mid \mathcal{K}(DZ)\mid \ar[rrrr]
      &&  && \mid \mathcal{K}(D(Z\setminus Y))\mid        }$$
are equal.
\end{lem}
\begin{proof}
This a straightforward consequence of Theorem~\ref{301} (1).
\end{proof}

Now, we are ready to give the the proof of Theorem~\ref{401}.

\begin{proof}(of Theorem~\ref{401})
We use the same trick as the proof of \cite[Theorem 3.1]{Gi}, the Riemann-Roch theorem without denominators. Firstly, for any $t$, we write down a square
$$\xymatrix{\mid \widehat{S}(\mathbb{A}_Y^1)\mid \ar[r]^-{c} \ar[d]^-{s_t^*} & \mid \mathcal{K}(\bigoplus_{p\geq0}\widetilde{D}^{2p-*}(\mathbb{A}_{Y_{\mu_n}}^1,p)[-1]_{R_n})\mid \ar[d]^-{s_t^*}\\
\mid \widehat{S}(Y)\mid \ar[r]^-{c_t} & \mid \mathcal{K}(\bigoplus_{p\geq0}\widetilde{D}^{2p-*}(Y_{\mu_n},p)[-1]_{R_n})\mid}$$
We claim that this square is commutative in the homotopy category $Ho(\mathcal{V})$. Recall that the map $c$ is defined as the composition ${\rm td}_g\circ \widetilde{{\rm ch}_g}$ (see Section 4.1), it is sufficient to show that $\widetilde{{\rm ch}_g}$ and ${\rm td}_g$ both commute with $s_t^*$.  $\widetilde{{\rm ch}_g}$ commutes with $s_t^*$ because the equivariant higher Bott-Chern forms is contravariant,  ${\rm td}_g$ commutes with $s_t^*$ because we endow all submanifolds of $W(\C)$ with the induced K\"{a}hler metrics and hence $\overline{N}_{W/{\mathbb{A}_Y^1}}\mid_{W_t}$ is isometric to $\overline{N}_{W_t/Y}$.
 
Next, denote $h:=(i_*^H)^{-1}\circ \gamma_c\circ i_*$, we claim that, for any $t$, the following diagram  
$$\xymatrix{\mid \widehat{S}(\mathbb{A}_Y^1)\mid \ar[r]^-{h} \ar[d]^-{s_t^*} & \mid \mathcal{K}(\bigoplus_{p\geq0}\widetilde{D}^{2p-*}(\mathbb{A}_{Y_{\mu_n}}^1,p)[-1]_{R_n})\mid \ar[d]^-{s_t^*}\\
\mid \widehat{S}(Y)\mid \ar[r]^-{h_t} & \mid \mathcal{K}(\bigoplus_{p\geq0}\widetilde{D}^{2p-*}(Y_{\mu_n},p)[-1]_{R_n})\mid}$$
is also commutative in $Ho(\mathcal{V})$.

According to Lemma~\ref{405} , $\gamma_c$ commutes with $s_t^*$. So it is sufficient to show that $i_*$ and $i_*^H$ also commute with $s_t^*$. $i_*$ commutes with $s_t^*$ because of the Tor-independence of the deformation diagram. $i_*^H$ commutes with $s_t^*$ because of the same reason and can be seen from the fact that the second quasi-isomorphism in \cite[Section 3 (3.2)]{Gi} (applied to $Y_{\mu_n}$) commutes with $s_t^*$ (see \cite[p. 236]{Gi}) and the fact that the immersions $\mathbb{A}_{Y_{\mu_n}}^1\to W_{\mu_n}$, $X_{\mu_n}\times\{t\}\to W_{\mu_n}$, $Y_{\mu_n}\times\{0\}\to \mathbb{P}(N_{X/Y}\oplus \mathcal{O}_Y)_{\mu_n}$ factor through the deformation to the normal cone construction with respect to $Y_{\mu_n}\to X_{\mu_n}$.

Since both K-theory and Deligne-Beilinson cohomology are homotopy invariant, there are homotopy equivalences
$$\xymatrix{p^*:\quad \mid \widehat{S}(Y)\mid \ar[r]^-{\sim} & \mid \widehat{S}(\mathbb{A}_Y^1)\mid}$$
and
$$\xymatrix{p_{\mu_n}^*:\quad \mid \mathcal{K}(\bigoplus_{p\geq0}D^{2p-*}(Y_{\mu_n},p)[-1]_{R_n})\mid \ar[r]^-{\sim} & \mid \mathcal{K}(\bigoplus_{p\geq0}D^{2p-*}(\mathbb{A}_{Y_{\mu_n}}^1,p)[-1]_{R_n})\mid.}$$
Moreover, notice that $s_t$ is a section of $p$ for any $t$, we conclude that $c_t=h_t$ in $Ho(\mathcal{V})$ if and only if $c=h$ in $Ho(\mathcal{V})$. Then the proof of Theorem~\ref{401} for general case reduces to the proof of Theorem~\ref{401} for the case of zero section embedding, which has been done in last subsection.
\end{proof}

\subsection{Byproduct: Riemann-Roch theorem for higher equivariant algebraic K-theory}
\label{sec:4.4}
Let us go back to the statement of Theorem~\ref{401}. The following square
$$\xymatrix{\mid \widehat{S}(Y)\mid \ar[r]^-{\sim}_-{i_*} \ar[d]_-{c} & Fj^* \ar[d]^-{\gamma_c} \\
\mid \mathcal{K}(\bigoplus_{p\geq0}\widetilde{D}^{2p-*}(Y_{\mu_n},p)[-1]_{R_n})\mid \ar[r]^-{\sim}_-{i_*^H} & Fj_H^*,}$$
which commutes up to homotopy, implies a commutative diagram on the level of homotopy groups
$$\xymatrix{K_*(Y,\mu_n) \ar[d]_-{\cong}^-{i_*} \ar[rr]^-{{\rm Td}_g^{-1}(N_{X/Y})\cdot{\rm ch}_g} && \bigoplus_{p\geq0}H_D^{2p-*}(Y_{\mu_n},\R(p))_{R_n} \ar[d]_-{\cong}^-{i_*^H} \\
K_{Y,*}(X,\mu_n) \ar[rr]^-{{\rm ch}_{Y,g}} && \bigoplus_{p\geq0}H_{Y_{\mu_n},D}^{2p-*}(X_{\mu_n},\R(p))_{R_n}. }$$
It is clear that this commutative diagram can be extended to the following form
$$\xymatrix{K_*(Y,\mu_n) \ar[d]^-{i_*} \ar[rr]^-{{\rm Td}_g^{-1}(N_{X/Y})\cdot{\rm ch}_g} && \bigoplus_{p\geq0}H_D^{2p-*}(Y_{\mu_n},\R(p))_{R_n} \ar[d]^-{i_*^H} \\
K_*(X,\mu_n) \ar[rr]^-{{\rm ch}_g} && \bigoplus_{p\geq0}H_D^{2p-*}(X_{\mu_n},\R(p))_{R_n}, }$$
which is the Riemann-Roch theorem for equivariant regulator maps in case of closed immersions.

If $Y,X$ lie in the category of complex algebraic manifolds and $\mu_n=\mu_1$, Burgos and Wang showed in \cite[Section 5]{BW} that ${\rm ch}_g$ equals the Beilinson's regulator map, then our Theorem~\ref{401} actually gives an analytic proof of \cite[Corollary 3.7.(1)]{Gi} which is a corollary of Gillet's Riemann-Roch theorem without denominators. For general case where $n>1$, Theorem~\ref{401} provides an equivariant extension of Gillet's result.

Now, let $B$ be another proper $\mu_n$-equivariant complex algebraic manifold and suppose that $f: X\to B$ is a $\mu_n$-equivariant smooth holomorphic map. For the non-equivariant case $n=1$, Roessler developed in \cite{Roe} a theory of analytic torsion for cubes of hermitian vector bundles and he used such analytic torsion theory to construct an explicit simplicial homotopy for square
$$\xymatrix{ \widehat{S}(X) \ar[d]^-{f_*} \ar[rr]^-{{\rm Td}(\overline{Tf})\bullet\widetilde{{\rm ch}}} && \mathcal{K}(\bigoplus_{p\geq0}D^{2p-*}(X,p)[-1]) \ar[d]^-{f_*^H} \\
\widehat{S}(B) \ar[rr]^-{\widetilde{{\rm ch}}} && \mathcal{K}(\bigoplus_{p\geq0}D^{2p-*}(B,p)[-1]), }$$
in which $\overline{Tf}$ is the relative tangent bundle with respect to $f: X\to B$ with a suitable smooth metric. Therefore, one obtains an analytic proof of Gillet's Riemann-Roch theorem for compact fibrations. Roessler's method also works in the equivariant case $n>1$, one just uses Bismut-Ma's equivariant analytic torsion form instead of its non-equivariant version due to Bismut-K\"{o}hler. In particular, one may construct an explicit simplicial homotopy for the square
$$\xymatrix{ \widehat{S}(X,\mu_n) \ar[d]^-{f_*} \ar[rr]^-{{\rm Td}_g(\overline{Tf})\bullet\widetilde{{\rm ch}_g}} && \mathcal{K}(\bigoplus_{p\geq0}D^{2p-*}(X_{\mu_n},p)[-1]_{R_n}) \ar[d]^-{f_*^H} \\
\widehat{S}(B,\mu_n) \ar[rr]^-{\widetilde{{\rm ch}_g}} && \mathcal{K}(\bigoplus_{p\geq0}D^{2p-*}(B_{\mu_n},p)[-1]_{R_n}) }$$
and hence a commutative diagram
$$\xymatrix{K_*(X,\mu_n) \ar[d]^-{f_*} \ar[rr]^-{{\rm Td}_g(Tf)\cdot{\rm ch}_g} && \bigoplus_{p\geq0}H_D^{2p-*}(X_{\mu_n},\R(p))_{R_n} \ar[d]^-{f_*^H} \\
K_*(B,\mu_n) \ar[rr]^-{{\rm ch}_g} && \bigoplus_{p\geq0}H_D^{2p-*}(B_{\mu_n},\R(p))_{R_n}. }$$
This can be regarded as a Riemann-Roch theorem for equivariant regulator maps in case of compact fibrations.

As usual, combining with the Riemann-Roch theorem for closed immersions, we get a complete Riemann-Roch theorem for projective morphisms.

\begin{thm}\label{406}
Let $X,B$ be two $\mu_n$-equivariant arithmetic schemes which are proper over an arithmetic ring, and let $f: X\to B$ be a $\mu_n$-projective morphism. Then the equivariant regulator maps fit into a commutative diagram
$$\xymatrix{K_*(X,\mu_n) \ar[d]^-{f_*} \ar[rr]^-{{\rm Td}_g(f)\cdot{\rm ch}_g} && \bigoplus_{p\geq0}H_D^{2p-*}(X_{\mu_n},\R(p))_{R_n} \ar[d]^-{f_*^H} \\
K_*(B,\mu_n) \ar[rr]^-{{\rm ch}_g} && \bigoplus_{p\geq0}H_D^{2p-*}(B_{\mu_n},\R(p))_{R_n}, }$$
where ${\rm Td}_g(f)$ is defined as ${j_{\mu_n}}^*{\rm Td}_g(Tp)\cdot{\rm Td}_g^{-1}(N_j)$ for any $\mu_n$-equivariant factorization
$$\xymatrix{f: X \ar[r]^-{j} & \mathbb{P}_B^r \ar[r]^-{p} & B.}$$
\end{thm}
\begin{proof}
Just notice that the equivriant algebraic K-theory is covariant functorial and the expression ${j_{\mu_n}}^*{\rm Td}_g(Tp)\cdot{\rm Td}_g^{-1}(N_j)$ is independent of the choice of factorization $f=p\circ j$.
\end{proof}

\section{Application: Arithmetic concentration theorem}
\label{sec:5}
In this section, we use localization sequence of higher equivariant K-groups to show an arithmetic concentration theorem. This theorem states that the higher equivariant arithmetic K-groups of an arithmetic scheme can be identified with the higher equivariant arithmetic K-groups of its fixed point subscheme after a suitable localization. To understand this localization, let us first introduce a kind of action on the higher equivariant arithmetic K-groups.

Let $X$ be a $\mu_n$-equivariant arithmetic scheme over an arithmetic ring $(D,\Sigma,F_\infty)$, and let $\overline{E}$ be a $\mu_n$-equivariant vector bundle on $X$ with smooth at infinity metric. Then $(\cdot)\otimes \overline{E}$ defines an exact functor from $\widehat{\mathcal{P}}(X)$ to itself, and hence a simplicial map from $\widehat{S}(X)$ to itself. According to the construction of $\widetilde{{\rm ch}_g}$, we have the following diagram
$$\xymatrix{ \widehat{S}(X) \ar[d]^-{\otimes\overline{E}} \ar[rr]^-{\widetilde{{\rm ch}_g}} && \mathcal{K}(\bigoplus_{p\geq0}\widetilde{D}^{2p-*}(X_{\mu_n},p)[-1]_{R_n}) \ar[d]^-{\bullet{\rm ch}_g^0(\vartheta(\overline{E}))} \\
\widehat{S}(X) \ar[rr]^-{\widetilde{{\rm ch}_g}} && \mathcal{K}(\bigoplus_{p\geq0}\widetilde{D}^{2p-*}(X_{\mu_n},p)[-1]_{R_n}), }$$
which is strictly commutative. So we get a group homomorphism $$\otimes\overline{E}: \widehat{K}_*(X,\mu_n)\to \widehat{K}_*(X,\mu_n)$$ such that $\otimes\overline{F}\circ\otimes\overline{E}=\otimes\big(\overline{E}\otimes\overline{F}\big)$ and $\otimes\overline{F}+\otimes\overline{E}=\otimes\big(\overline{F}\oplus\overline{E}\big)$. We formally define $\otimes(-\overline{E})=-\otimes\overline{E}$.

Denote $R(\mu_n):=K_0(\Z,\mu_n)\cong \Z[\Z/{n\Z}]\cong \Z[T]/{(1-T^n)}$, let $\overline{I}$ be the $\mu_n$-comodule whose term of degree $1$ is $\Z$ endowed with the trivial metric and whose other terms are $0$. We then make $\widehat{K}_*(X,\mu_n)$ an $R(\mu_n)$-module under the assignment $T\to \overline{I}$.

Let $X$ be proper over $D$ and let $Y$ be a $\mu_n$-equivariant closed subscheme in $X$ with complement $U:=X\setminus Y$, then the $R(\mu_n)$-actions on $\widehat{K}_*(X,\mu_n)$ and on $\widehat{K}_*(U,\mu_n)$ induce a $R(\mu_n)$-action on $\widehat{K}_{Y,*}(X,\mu_n)$. It is clear that the localization sequence~(\ref{a1}) is compatible with these $R(\mu_n)$-actions, and we claim that the isomorphism $i_*: \widehat{K}_*(Y,\mu_n)\cong \widehat{K}_{Y,*}(X,\mu_n)$ commutes with the $R(\mu_n)$-actions. In fact, this can be reduced to the case of zero section embedding $i_0: Y\to P$ in which one can readily check the compatibility using the fact that $i_*: \widehat{K}_*(Y,\mu_n)\to \widehat{K}_*(P,\mu_n)$ is injective.

Now, for any hermitian $\mu_n$-equivariant vector bundle $\overline{E}$ on $X$, we write $\otimes\big(\lambda_{-1}(\overline{E})\big)$ for the expression
$$\otimes\big(\overline{\mathcal{O}}_X\big)-\otimes\big(\Lambda^1\overline{E}\big)+\otimes\big(\Lambda^2\overline{E}\big)+\cdots+(-1)^m\otimes\big(\Lambda^m\overline{E}\big)$$
where $m$ is the rank of $E$. We present a refinement of the equivariant self intersection formula for regular closed immersions.

\begin{thm}\label{501}(arithmetic self intersection formula)
Let notations and assumptions be as in Theorem~\ref{401}, then we have
$$i^*i_*(\cdot)=(\cdot)\otimes\big(\lambda_{-1}(\overline{N}_{X/Y}^\vee)\big)$$ as group endomorphisms of $\widehat{K}_*(Y,\mu_n)$.
\end{thm}
\begin{proof}
Since the K-theory and the Deligne-Beilinson cohomology are $\mathbb{A}^1$-homotopy invariant, we may again use the deformation to the normal cone construction to reduce our proof to the case of zero section embedding $i_0: Y\hookrightarrow P=\mathbb{P}(N\oplus\mathcal{O}_Y)$. Actually, we have $i^*i_*=i_0^*{i_0}_*$. Then we write $i_0=i$ for short. By taking the $\mu_n$-averages, we may suppose that all metrics are $\mu_n$-invariant. Regard the K-theory spaces as $\Omega-$spectra, then $\otimes(\lambda_{-1}\overline{N}^\vee)$ defines an infinite loop map $\mid\widehat{S}(Y,\mu_n)\mid\to \mid\widehat{S}(Y,\mu_n)\mid$. Concerning the construction of the embedding morphism $i_*$ and the fact that $i^*\overline{Q}$ is isometric to $\overline{N}$, the map $i^*i_*$ coincides with $\otimes(\lambda_{-1}\overline{N}^\vee)$.

Next, notice that
$$i_{\mu_n}^*({i_{\mu_n}}_!)'\big({\rm Td}_g^{-1}(\overline{N})\big)={\rm ch}_g(\sum_{j=0}^{{\rm rk}N}(-1)^j\wedge^j\overline{N}^\vee)$$
and that $i_H^*\circ (i^H_*)'$ is homotopic to $i_H^*\circ i^H_*$ with an explicit homotopy
$$i_H^*\circ \beta\circ \widetilde{e}(P_0,\overline{Q}_g\mid_{P_0},i_{\mu_n}),$$
where $\beta$ is a chosen simplicial homotopy inverse
$$\mathcal{K}(\bigoplus_{p\geq0}{'D}^{2p-*}(P_{\mu_n},p)[-1]_{R_n})\to \mathcal{K}(\bigoplus_{p\geq0}{D}^{2p-*}(P_{\mu_n},p)[-1]_{R_n}).$$
So we are left to show that $i_H^*\circ \beta\circ {\rm Td}_g^{-1}(\overline{N})\bullet\widetilde{e}(P_0,\overline{Q}_g\mid_{P_0},i_{\mu_n})$ coincides with the homotopy $i_H^*\circ \beta\circ T_g\big(K(\overline{\mathcal{O}_Y},\overline{N})\big)$. By projection formula and the definition $\widetilde{e}(P_0,\overline{Q}_g\mid_{P_0},i_{\mu_n})={\rm Td}_g(\overline{Q})\bullet T_g\big(K(\overline{\mathcal{O}_Y},\overline{N})\big)$, this is equivalent to show that
$$i_H^*\circ \beta\circ \pi_P^*{\rm Td}_g^{-1}(\overline{N})\circ \widetilde{e}(P_0,\overline{Q}_g\mid_{P_0},i_{\mu_n})=i_H^*\circ \beta\circ {\rm Td}_g^{-1}(\overline{Q})\circ \widetilde{e}(P_0,\overline{Q}_g\mid_{P_0},i_{\mu_n}).$$
Further, the homotopy connecting $\beta\circ \pi_P^*{\rm Td}_g^{-1}(\overline{N})$ and $\pi_P^*{\rm Td}_g^{-1}(\overline{N})\circ \beta$ is the same as the homotopy connecting $\beta\circ {\rm Td}_g^{-1}(\overline{Q})$ and ${\rm Td}_g^{-1}(\overline{Q})\circ \beta$. Then our statement follows from the fact that $i^*{\rm Td}_g^{-1}(\overline{Q})={\rm Td}_g^{-1}(\overline{N})=i^*\pi_P^*{\rm Td}_g^{-1}(\overline{N})$. So we are done.
\end{proof}

Specifying the natural inclusion $i: X_{\mu_n}\hookrightarrow X$, this arithmetic self intersection formula plays a crucial role in looking for the inverse of $i_*: \widehat{K}_*(X_{\mu_n},\mu_n)\to \widehat{K}_*(X,\mu_n)$ (after a suitable localization). We now clarify this localization. Let $\rho$ be a prime ideal in $R(\mu_n)$ such that the elements $1-T^k$ for $k=1,\cdots,n-1$ are not contained in $\rho$. For instance, one may choose $\rho$ to be the kernel of the canonical morphism $\Z[T]/{(1-T^n)}\to \Z[T]/{(\Phi_n)}$ where $\Phi_n$ stands for the $n$-th cyclotomic polynomial.
Then the arithmetic concentration theorem can be formulated as follows.

\begin{thm}\label{502}
Let $X$ be a $\mu_n$-equivariant arithmetic scheme which is proper over an arithmetic ring, and let $i: X_{\mu_n}\hookrightarrow X$ be the closed immersion from the fixed point subscheme $X_{\mu_n}$ to $X$. Then the group homomorphism $$i_*: \widehat{K}_m(X_{\mu_n},\mu_n)\to \widehat{K}_m(X,\mu_n)$$ induces an isomorphism $i_*: \widehat{K}_m(X_{\mu_n},\mu_n)_{\rho}\cong \widehat{K}_m(X,\mu_n)_{\rho}$ for any $m\geq 1$. Denote by $\overline{N}$ the normal bundle of $X_{\mu_n}$ in $X$ endowed with the metric induced from a K\"{a}hler metric of $X(\C)$. Then the map $$\otimes \lambda_{-1}(\overline{N}^\vee): \widehat{K}_m(X_{\mu_n},\mu_n)_{\rho}\to \widehat{K}_m(X_{\mu_n},\mu_n)_{\rho}$$
admits a formal inverse $\otimes \lambda_{-1}^{-1}(\overline{N}^\vee)$ and $\otimes \lambda_{-1}^{-1}(\overline{N}^\vee)\circ i^*$ is the inverse of
$$i_*: \widehat{K}_m(X_{\mu_n},\mu_n)_{\rho}\cong \widehat{K}_m(X,\mu_n)_{\rho}.$$
\end{thm}
\begin{proof}
Denote by $j: U\hookrightarrow X$ the complement of $X_{\mu_n}$ in $X$. Then we have a long exact sequence
$$\xymatrix{\cdots\ar[r] & \widehat{K}_{m}(X_{\mu_n},\mu_n)_{\rho} \ar[r]^-{i_*} & \widehat{K}_m(X,\mu_n)_{\rho} \ar[r]^-{j^*} & \widehat{K}_m(U,\mu_n)_{\rho} \ar[r] & \widehat{K}_{m-1}(X_{\mu_n},\mu_n)_{\rho} \ar[r] & \cdots}$$
ending with
$$\xymatrix{\cdots\ar[r] & \widehat{K}_{1}(X_{\mu_n},\mu_n)_{\rho} \ar[r]^-{i_*} & \widehat{K}_1(X,\mu_n)_{\rho} \ar[r]^-{j^*} & \widehat{K}_1(U,\mu_n)_{\rho}.}$$

Since $U_{\mu_n}=\emptyset$, according to the definition of $\widehat{K}_*$, we know that $\widehat{K}_*(U,\mu_n)\cong K_*(U,\mu_n)$ and hence $\widehat{K}_*(U,\mu_n)_{\rho}\cong K_*(U,\mu_n)_{\rho}$. By Thomason's algebraic concentration theorem \cite[Th\'{e}or\`{e}me 2.1]{Th2}, $K_*(U,\mu_n)_{\rho}=0$, so $\widehat{K}_*(U,\mu_n)_{\rho}=0$. This implies that $$i_*: \widehat{K}_*(X_{\mu_n},\mu_n)_\rho\to \widehat{K}_*(X,\mu_n)_\rho$$ is an isomorphism.

In \cite[Lemma 4.5]{KR}, K\"{o}hler and Roessler constructed a formal inverse of $\lambda_{-1}(\overline{N}^\vee)$ in $\widehat{K}_0(X_{\mu_n},\mu_n)_{\rho}$. This formal inverse represents a finite polynomial of hermitian vector bundles with coefficients in $R(\mu_n)_{\rho}$, and hence it provides a formal inverse of the map $\otimes \lambda_{-1}(\overline{N}^\vee)$.

At last, the statement that $\otimes \lambda_{-1}^{-1}(\overline{N}^\vee)\circ i^*$ is the inverse of
$i_*: \widehat{K}_m(X_{\mu_n},\mu_n)_{\rho}\cong \widehat{K}_m(X,\mu_n)_{\rho}$ follows from Theorem~\ref{501} and the fact that $i_*: \widehat{K}_m(X_{\mu_n},\mu_n)_{\rho}\cong \widehat{K}_m(X,\mu_n)_{\rho}$ is already an isomorphism.
\end{proof}

\textbf{Acknowledgements:}
The author would like to thank Damian Roessler for his careful reading of an early version of this paper and for his valuable comments. 


\begin{thebibliography}{}
\bibitem[Bi]{Bi}
Bismut J.-M., \emph{Equivariant immersions and Quillen metrics},
J. Differential Geom. \textbf{41}(1995), 53-157.

\bibitem[BK]{BK}
Bousfield A. K. and Kan D. M., \emph{Homotopy limits, completions and localizations},
Lecture Notes in Math. \textbf{304}, Springer 1972.

\bibitem[BKK]{BKK}
Burgos J. I., Kramer J. and K\"{u}hn U., \emph{Cohomological arithmetic Chow rings},
J. Inst. Math. Jussieu \textbf{6}(2007), 1-172.

\bibitem[BO]{BO}
Bloch S. and Ogus A., \emph{Gersten's conjectures and the homology of schemes},
Ann. Sci. Ecole Norm. Sup. \textbf{7}(1974),
4\`{e}me s\'{e}rie, 181-202.

\bibitem[Bu1]{Bu1}
Burgos J. I., \emph{Aritjmetic Chow rings and Deligne-Beilinson cohomology},
J. Alg. Geom. \textbf{6}(1997), 335-377.

\bibitem[BW]{BW}
Burgos J. I. and Wang S., \emph{Higher Bott-Chern forms and Beilinson's regulator},
Invent. Math. \textbf{132}(1998), 261-305.

\bibitem[De]{De}
Deligne P., \emph{Le d\'{e}terminant de la cohomologie},
Conterporary Mathematics \textbf{67}(1987), 93-177.

\bibitem[Fe]{Fe}
Feliu E., \emph{Adams operations on higher arithmetic K-theory},
Publ. RIMS Kyoto Univ. \textbf{46}(2010), 115-169.

\bibitem[Gi]{Gi}
Gillet H., \emph{Riemann-Roch theorems for higher algebraic K-theory},
Advances in Math. \textbf{40}(1981), 203-289.

\bibitem[GS1]{GS1}
Gillet H. and Soul\'{e} C., \emph{Arithmetic intersection theory},
Publ. Math. IHES \textbf{72}(1990), 94-174.

\bibitem[GS2]{GS2}
Gillet H. and Soul\'{e} C., \emph{Characteristic classes for
algebraic vector bundles with hermitian metrics I, II}, Ann. of
Math. \textbf{131}(1990), 163-203 and 205-238.

\bibitem[HS]{HS}
Holmstrom A. and Scholbach J., \emph{Arakelov motivic cohomology I}, to appear in Journal of Algebraic Geometry, available on arXiv:1012.2523v3[math.NT].

\bibitem[Ja]{Ja}
Jannsen U., \emph{Deligne homology, Hodge-D-conjecture and motives}, in \emph{Beilinson's conjectures on special values of L-functions},
Perspectives in Math. vol. \textbf{4}, Academic Press, 1988, 305-372.

\bibitem[Ko]{Ko}
K\"{o}ck B., \emph{The Lefschetz theorem in higher equivariant $K$-theory},
Comm. Algebra \textbf{19}(1991), 3411-3422.

\bibitem[KR]{KR}
K\"{o}hler K.  and Roessler D., \emph{A fixed point formula of
Lefschetz type in Arakelove geometry I: statement and proof}, Inventiones Math. \textbf{145}(2001), no.2, 333-396.

\bibitem[Le]{Le}
Lelong P., \emph{Int\'{e}gration sur un ensemble analytique
complexe}, Bull. Soc. Math. France \textbf{95}(1957), 239-262.

\bibitem[MC]{MC}
McCarthy R., \emph{A chain complex for the spectrum homology of the algebraic K-theory of an exact category},
Algebraic K-theory, Fields Inst. Commun. \textbf{16}, Amer. Math. Soc., Providence, 1997, 199-220.

\bibitem[MP]{MP}
May J. P. and Ponto K., \emph{More Concise Algebraic Topology: Localization, completion, and model categories}, Chicago Lectures in Math., Univ. Chicago Press, London, 2012.

\bibitem[Mi]{Mi}
Milnor J., \emph{On spaces having the homotopy type of a CW complex},
Trans. A.M.S. \textbf{90}(1959), 272-280.

\bibitem[Roe]{Roe}
Roessler D., \emph{Analytic torsion for cubes of vector bundles and Gillet's Riemann-Roch theorem},
J. Algebraic Geom. \textbf{8}(1999), 497-518.

\bibitem[Sch]{Sch}
Schechtman V. V., \emph{On the delooping of Chern character and Adams operations}, in \emph{K-theory, arithmetic and geometry}, LNM \textbf{1289}, Springer-Verlag, 1987, 265-319.

\bibitem[S1]{S1}
Scholbach J., \emph{Arakelov motivic cohomology II}, to appear in Journal of Algebraic Geometry, available on arXiv:1205.3890v2[math.AG].

\bibitem[So]{So}
Soul\'{e} C. et als., \emph{Lectures on Arakelov geometry},
Cambridge Studies in advanced mathematics \textbf{33}, Cambridge university Press, 1992.

\bibitem[T1]{T1}
Tang S., \emph{Concentration theorem and relative fixed point
formula of Lefschetz type in Arakelov geometry}, J. reine angew. Math., \textbf{665}(2012), 207-235.

\bibitem[Ta]{Ta}
Takada Y., \emph{Higher arithmetic K-theory}, Publ. Res. Inst. Math. Sci. \textbf{41}(2005), 599-681.

\bibitem[Th1]{Th1}
Thomason R., \emph{Algebraic K-theory of group scheme actions}, in \emph{Algebraic Topology and Algebraic K-theory}, Ann. of Math. Stud. \textbf{113}, Princeton Univ. Press, Princeton, 1987, 539-563.

\bibitem[Th2]{Th2}
Thomason R., \emph{Une formule de Lefschetz en K-th\'{e}orie
\'{e}quivariante alg\'{e}brique}, Duke Math. J. \textbf{68}(1992),
447-462.

\bibitem[TT]{TT}
Thomason R.  and Trobaugh T., \emph{Higher algebraic K-theory of schemes and of derived categories}, in \emph{The Grothendieck Festschrift, Vol. III}, Progress in Mathematics \textbf{88}, Birkh\"{a}user, Boston, MA(1990), 247-435.

\bibitem[Wal]{Wal}
Waldhausen F. , \emph{Algebraic K-theory of spaces}, in \emph{Algebraic and geometric topology},
LNM \textbf{1126}, Springer-Berlin, 1980, 318-419.

\end{thebibliography}


\hspace{5cm} \hrulefill\hspace{5.5cm}

Beijing Center for Mathematics and Information Interdisciplinary Sciences

School of Mathematical Sciences, Capital Normal University

West 3rd Ring North Road 105, 100048 Beijing, P. R. China

E-mail: shun.tang@outlook.com

\end{document}